\documentclass[11pt]{article}
\usepackage{amsmath}
\usepackage[a4paper]{geometry}
\usepackage{amsfonts}
\usepackage{algorithm}
\usepackage{algorithmicx}
\usepackage{algpseudocode}
\usepackage{enumerate}
\usepackage{enumitem}
\usepackage[section]{placeins}
\usepackage{amsthm}
\usepackage{xcolor}
\usepackage{graphicx}
\graphicspath{{figures/}}
\usepackage{bm}
\usepackage{multirow}
\usepackage{lipsum}
\usepackage[caption=false]{subfig}
\usepackage{epstopdf}
\ifpdf
  \DeclareGraphicsExtensions{.eps,.pdf,.png,.jpg}
\else
  \DeclareGraphicsExtensions{.eps}
\fi

\usepackage{enumitem}
\setlist[enumerate]{leftmargin=.5in}
\setlist[itemize]{leftmargin=.5in}

\theoremstyle{plain}

\newtheorem{theorem}{Theorem}
\newtheorem{lemma}[theorem]{Lemma}
\newtheorem{proposition}[theorem]{Proposition}
\newtheorem{corollary}[theorem]{Corollary}
\newtheorem{remark}{Remark}
\newtheorem{hypothesis}{Hypothesis}

\def\mou{\color{black}}
\begin{document}

\title{Robust exploratory mean-variance problem with drift uncertainty}
\author{Chenchen Mou, Weiwei Zhang and Chao Zhou}
\maketitle

\begin{abstract}
    We solve a min-max problem in a robust exploratory mean-variance problem with drift uncertainty in this paper. It is verified that robust investors choose the Sharpe ratio with minimal $L^2$ norm in an admissible set. A reinforcement learning framework in the mean-variance problem provides an exploration-exploitation trade-off mechanism; if we additionally consider model uncertainty, the robust strategy essentially weights more on exploitation rather than exploration and thus reflects a more conservative optimization scheme. Finally, we use financial data to backtest the performance of the robust exploratory investment and find that the robust strategy can outperform the purely exploratory strategy and resist the downside risk in a bear market.
\end{abstract}

\begin{keywords}
  model uncertainty, exploratory mean-variance analysis, robustness, exploration and exploitation, min-max problem
\end{keywords}

\section{Introduction}
In this paper, we focus on an exploratory mean-variance problem with drift uncertainty. The exploratory version of the mean-variance problem mainly refers to \cite{WangHaoran2} and the concept of model uncertainty stems from \cite{JinHanqing2015}. The classical continuous-time mean-variance problem illustrates that the optimal strategy to balance the wealth state in different assets is represented by market parameters (mean return rate $\mu$, volatility matrix $\sigma$, interest rate $r$) and optimal Lagrangian multiplier $\omega$. Practical implementation of the mean-variance model directly relates to the level of market parameters, whereas in real financial markets accurate values of model parameters are unknown. It is common to estimate parameters through historical market data by a variety of calibration techniques \cite{Berestycki2002,Merton1980}; however, calibration results often disperse among empirical methods. As a result, it is quite hard to pick out a consensus estimation to match most of the market scenarios.

Recently, a novel approach is intentionally designed by Wang \cite{WangHaoran2,WangHaoran3} to solve the mean-variance portfolio problem via a reinforcement learning framework. More specifically, investors' decision processes are replaced by relaxed controls, and a new trade-off relationship between exploration and exploitation appears. Exploration means that investors are encouraged to explore the optimal strategy using distributional rules, and each feedback control is a classical control sampled from the distributional rule. Exploitation means that with repetitive sampling rounds, investors gradually recover the optimal strategy by sample distributions rather than unknown market parameters $\mu, \sigma$, then the optimization is based on the information of control samples. Accordingly, the classical mean-variance problem has converted to an exploratory mean-variance problem. For more details about exploratory problem settings, we refer readers to \cite{WangHaoran1}. The critical advantage of exploratory idea is that the convergence of control samples directly guides us how to optimize the portfolio without any knowledge about the accurate value of $\mu$ and $\sigma$; so that we are able to skip the troubles brought by calibration as well as estimation errors. \cite{WangHaoran2} shows that optimal distributional rule of control is Gaussian and the value function can be solved explicitly and represented by a set of redefined parameters. At the end, optimal parameter values are trained by a reinforcement learning algorithm and feeding the real market data.

The key algorithm ``ENT-MV" in \cite{WangHaoran2} implicates that, even though the portfolio optimization through the reinforcement learning algorithm does not directly involve $\mu$ and $\sigma$, new parameters in ``ENT-MV" algorithm to be optimized are merely recombination of original parameters $\mu$ and $\sigma$. Once the system is trained to be in its optimality, the algorithm essentially provides an ``optimal" estimation of $\mu$ and $\sigma$ meanwhile. The methodology behind the ``ENT-MV" algorithm is nothing but replacing the traditional parameter estimation by a learning-based one so that learned parameters guarantee the optimality of the portfolio problem. However, under the viewpoint of model uncertainty, this idea suffers the same drawback as the traditional statistical estimation: the estimation is completely driven by data, while data are not always effective, or even effective past data may wrongly predict the future.

In order to fill this gap, our purpose is to add model uncertainty into the exploratory mean-variance problem and find out the robust solution under model uncertainty. Model uncertainty is intrigued by the fact that investors often fail to have a complete knowledge about the model, herein we admit the uncertainty driven by unknown parameters (drift, volatility etc.) and attempt to consider the worst-case scenarios among all potential combinations of parameters in a confidence region. Optimal solutions of portfolio problems by choosing the worst market parameters are so-called robust solutions. Model uncertainty has been considered in pricing problem since \cite{Avellaneda1995,Lyons1995}, and it was \cite{Hansen2001} who first brought the idea into portfolio problems. Later on, \cite{ChenEpstein2002,Goldfarb2003} respectively extended robust problems to continuous-time and single-period models. Other related works consider various kinds of utility functions or model settings combining with uncertainty and robust solutions, \cite{Skiadas2003,Gundal2005,JinHanqing2015}. Notice that literatures above merely considered drift uncertainty. When volatility uncertainty is involved, we refer to \cite{LinRiedel2014,Matoussi2015,Tevzadze2013} for detailed description. It is known that the drift is the main source of uncertainty because the drift is the hardest part to be estimated precisely. In order to simplify the model and consider the most crucial factor, we only discuss the drift uncertainty in the current work.

We investigate on a robust exploratory mean-variance problem in this paper. It is reasonable to suspect that parameters calibration through market data are misspecified, so we add model uncertainty to the original problem in \cite{WangHaoran2}. Among all the unknown market parameters we only consider the drift uncertainty here and express it by risk premium $\varrho_t$. Our purpose is to find the ``worst" $\varrho_t$ in an admissible closed convex set and the ``best" control distribution based on the ``worst" scenario. We call it the robust solution of the exploratory mean-variance portfolio optimization problem. Our model setting inherits \cite{WangHaoran3} in exploratory part and \cite{JinHanqing2015} in drift uncertainty part. It was proved in \cite{WangHaoran3} that with an exploration term, the optimal control distribution is Gaussian and the value function can be solved explicitly. When robustness is induced, the exploratory optimization becomes a min-max problem, so we find out the saddle point $(\varrho_t^*, \pi_t^*)$ which can switch the $\min$ and $\max$ and solve the robust value function simultaneously. In this case, the ``worst" $\varrho_t^*$ coincidentally achieves the minimal $L^2$ norm in its admissible set. We further discuss the effect of the robust strategy comparing with misspecified purely exploratory strategies. Due to the appearance of the exploration term, the optimal exploratory strategy should make a balance between ``exploration" (trying new strategies to obtain information from a larger range) and ``exploitation" (optimize the main target of reducing the terminal variance). It is interesting to notice that a merely ``exploitation" targeted investor can improve the terminal variance result from the optimal exploration strategy by choosing a misspecified $\hat{\varrho}$ which is smaller than the genuine market scenario $\varrho$. This is essentially an adjustment of the weight between exploration and exploitation. The phenomenon also matches the behavior of a robust investor: a more conservative investor reduce his/her market viewpoint $\varrho$ to obtain more opportunity to reduce the terminal variance and emphasize exploitation rather than exploration. We verify the variance reduction effect by a numerical simulation and test the performance of robust investment by feeding different financial data and calibrating the parameters.

The paper is organized as follows. We introduce the exploratory mean-variance problem and induce drift uncertainty into the model {\mou in Section \ref{section41}}. Then the robust strategy, a min-max problem's solution and its associated saddle-point are given {\mou in Section \ref{section42}}. {\mou In Section \ref{section43} we discuss }the effect of the robust strategy. The parameter $\varrho$ is calibrated and the performance of the robust strategy with real market data is presented in Section \ref{section44}. {\mou In Section \ref{section45} we summarize all the results in the paper. Finally, in Appendix \ref{AppendixB}, we finish some technical proofs which were postponed in the previous sections.}

\section{Portfolio models} \label{section41}

Assume there are $d$ risky assets in the financial market. Given a filtered probability space $(\Omega, \mathcal{F}, \{\mathcal{F}_t\}_{t\geq 0}, \mathbb{P})$, we define a $d$-dimensional $\mathcal{F}_t$-adapted Brownian motion $W_t:=(W_t^1,W_t^2,\cdots,W_t^{d})'$, where $'$ stands for the matrix transpose. Assume the stock market $S_t:=(S_t^1,S_t^2,\cdots, S_t^d)'\in\mathbb R^d$ follows {\mou a} geometric Brownian motion
\begin{equation*}
  dS_t^i = \mu_t^iS_t^idt +S_t^i\sum_{i=1}^d\sigma^{ij}_tdW_t^j\quad \text{on $[0,T]$},
\end{equation*}
with $S_0^i:=s_0^i>0$, where {\mou $\sigma_t:=\{\sigma_t^{ij}\}_{1\leq i,j\leq d}\in\mathbb R^{d\times d}$ is a deterministic volatility matrix whose inverse exists, and {$\mu_t:=(\mu_t^1,\mu_t^2,\cdots,\mu_t^d)'\in\mathbb R^d$ is an $\mathcal{F}_t$-adapted random drift which brings uncertainty to the model.}} Let $r_t>0$ be a deterministic risk-free rate. The stock price can be rewritten in terms of a risk premium process $\varrho_t:=\sigma_t^{-1}(\mu_t-r_t(1,1,\cdots,1)')\in\mathbb R^d$. We transfer the model uncertainty into the uncertain risk premium $\varrho_t$ for convenience although the model uncertainty stems from $\mu_t$. An investor's control process $v_t\in\mathbb{R}^d$ is randomized to present exploration and its density function is given by $\pi_t(v)\in\mathcal{P}(\mathbb R^d)$ where $\mathcal{P}(\mathbb R^d)$ stands for the set of density functions of absolutely continuous probability measures with respect to the Lebesgue measure on $\mathbb R^d$. In this case, a discounted self-financing wealth process $X_t^\pi\in\mathbb R$ with its initial wealth state $x_0\in\mathbb R$ has the following dynamic
\begin{equation}
  dX_t^\pi = \left( \int_{\mathbb{R}^d} {\varrho_t}'\sigma_t v\pi_t(v) dv\right)dt + \sqrt{\int_{\mathbb{R}^d} v'{\sigma_t}'\sigma_tv\pi_t(v)dv}\;dW_t\quad\text{on $[0,T]$}.  \label{wealth}
\end{equation}
Following the setting of \cite{WangHaoran3}, without model uncertainty, the classical exploratory mean-variance problem is to solve the value function
\begin{equation}
  V(x_0,0) = \min_{{\{\pi_t\}_t}} \mathbb{E}\left[(X_T^\pi-\omega)^2+c\int_0^T\int_{\mathbb{R}^d}\pi_t(v)\ln\pi_t(v)dvdt\right]-(\omega-l)^2, \label{main0}
\end{equation}
where $\omega$ is the Lagrangian multiplier under the optimal control, $l$ is the default target of the wealth expectation at maturity, $c>0$ is the exploration intensity and the additional term\footnote{We denote it as entropy term henceforth.} $$\int_0^T\int_{\mathbb{R}^d}\pi_t(v)\ln\pi_t(v)dvdt<0$$ is the opposite of Shannon-entropy. Optimizing the exploration is the same as maximizing the Shannon-entropy, and thus minimizing this additional term in \eqref{main0}. Intuitively speaking, an larger $c$ means more exploration: in particular, $c=0$ reduces the minimization problem \eqref{main0} into a standard mean-variance problem, where the density function of the optimal control is degenerated, and the probability measure with respect to exploration is a Dirac measure. When $c$ is very large, exploitation is negligible; we only optimize the exploration term. In this case, the optimal density function is Gaussian because Gaussian distribution family maximizes the Shannon-entropy. In general, the entropy term in \eqref{main0} encourages an investor to explore among the admissible controls and diversify his/her feedback strategies. As a result, the exploratory mean-variance problem becomes a trade-off between exploitation and exploration.

The appearance of uncertainty forces investors to consider the worst case in a range of models although the optimal strategy is selected. In this case, the optimization problem becomes a min-max problem
\begin{equation}
  V(x_0,0) = {\min_{\{\pi_t\}_t}\max_{\{\varrho_t\}_t}\ } \mathbb{E}\left[(X_T^\pi-\omega)^2+c\int_0^T\int_{\mathbb{R}^d}\pi_t(v)\ln\pi_t(v)dvdt\right]-(\omega-l)^2.
  \label{main1}
\end{equation}

{\mou The following we make the precise assumptions on parameters and give admissible sets of $\{\pi_t\}_t$ and $\{\rho_t\}_t$ for the min-max problem \eqref{main1}:}
\begin{hypothesis}\hspace{1em} \\\vspace{-1em}
\begin{itemize}
  \item $\int_0^T |\mu^i_t|\; dt <\infty\;\  \mathbb{P}-a.s.$, \hspace{0.2em} $\int_0^T |\sigma^{ij}_t|^2\; dt <\infty\,\,\, \mathbb{P}-a.s.$, \hspace{0.2em} $\int_0^T |r_t|\;dt<\infty$, for $i,j=1,\cdots, d$.
 \item $\exists \epsilon>0$, $\forall t>0$, $\sigma_t\sigma'_t >\epsilon I_d$, where $I_d$ is the $d$-dimensional identity matrix.
 {\mou  \item Let $\Xi$ be a closed convex subset of $\mathbb R^{d\times d}\setminus\{0\}$; the admissible set $\mathfrak{\Xi}_t$ is the set of processes $\{\varrho_s\}_{s\in[t,T]}$ such that $\varrho_s\in \Xi$ for all $s\in[t,T]$.
  \item Let $A\subset\mathcal{P}(\mathbb R^d)$; the admissible set $\mathcal{A}_t$ is the set of processes $\{\pi_s\}_{s\in[t,T]}$ such that $\pi_s\in\Xi$ for all $s\in[t,T]$.}
\end{itemize}
\end{hypothesis}

\begin{remark} \label{convention}
  $\Xi$ is closed convex and $0\notin\Xi$ imply that all admissible $\varrho_s:=(\varrho_s^{1},\varrho_s^{2}\cdots,\varrho_s^{d})$ are either positive or negative {\mou for all $s\in[t,T]$. By convention we assume that $\varrho_s^{i}>0$ for all $s\in[t,T]$ and $i=1,\cdots,d$ if $\{\varrho_s\}_{s}\in\mathfrak{\Xi}_t$ for any $t\in[0,T]$. }This assumption is the same as the one in \cite{JinHanqing2015}.
\end{remark}

Our target is to solve \eqref{main1} and prove a saddle point property for the problem.

\section{Optimal solution to robust exploratory mean-variance problem}  \label{section42}

We derive an explicit result on a saddle point property for \eqref{main1} and solve the exploratory mean-variance problem with drift uncertainty in this section. The dynamic programming argument shows that
\begin{align*}
  V(t,x) &= {\min_{\{\pi_s\}_s\in\mathcal{A}_t}\max_{\{\varrho_s\}_s\in\mathfrak{\Xi}_t}} \mathbb{E}\left[(X_T^\pi-\omega)^2+c\int_t^T\int_{\mathbb{R}^d}\pi_s(v)\ln\pi_s(v)dvds| X_t^\pi=x \right]-(\omega-l)^2\\
  &={\min_{\{\pi_s\}_s\in\mathcal{A}_t}\max_{\{\varrho_s\}_s\in\mathfrak{\Xi}_t}} \mathbb{E}^{t,x} \left[V(t+\triangle t, X_{t+\triangle t}^\pi) +c\int_t^{t+\triangle t}\int_{\mathbb{R}^d}\pi_s(v)\ln\pi_s(v)dvds\right],
\end{align*}
so the HJB equation of the robust exploratory mean-variance problem is
\begin{equation}
\begin{aligned}
  &\min_{\pi\in A} \max_{\varrho\in\Xi} \left\{V_t + \int_{\mathbb{R}^d}\left[ \dfrac{1}{2}v'{\mou \sigma_t'\sigma_t }v V_{xx} + \varrho'{\mou \sigma_t} v V_x +c\ln\pi(v)\right]\pi(v)dv\right\}=0 \\
  & V(T,x) = (x-\omega)^2-(\omega-l)^2.
\end{aligned}\label{HJB4-0}
\end{equation}

We split our solution into two steps. First, for any fixed specific market scenario where the model uncertainty is absent, we solve a classical exploratory mean-variance problem with a fixed drift. Then we find a ``worst" scenario in the admissible set and prove that the optimal policy under the ``worst" scenario is an equilibrium pair as well as the solution of \eqref{HJB4-0}. Two steps are presented in details in Subsection \ref{ClassicHJB} and Subsection \ref{RobustSolution} respectively.
\subsection{Classic exploratory solution by a HJB approach} \label{ClassicHJB}

First we consider the classical solution of the exploratory problem without model uncertainty, which follows the arguments as \cite{WangHaoran3}. For any $\rho:=\{\rho_t\}_{t}\in \mathfrak{\Xi}_0$ and $\pi:=\{\pi_t\}_{t}\in \mathcal{A}_0$, we define
\begin{equation*}
  M(\varrho,\pi):=\mathbb{E}\left[(X_T^\pi-\omega)^2+c\int_0^T\int_{\mathbb{R}^d}\pi_t(v)\ln\pi_t(v)dvdt\right]-(\omega-l)^2.
\end{equation*}

\begin{proposition}[Theorem 1 in \cite{WangHaoran3}] \label{Theorem41}
 {\mou With any fixed $\varrho:=\{\varrho_t\}_t\in\mathfrak{\Xi}_0$,} the optimal density function $\pi_t^*(v;\varrho)$ with respect to problem \eqref{main0} is Gaussian, and the value function is obtained by
\begin{align*}
  M(\varrho,\pi^*(\cdot;\varrho))=&\dfrac{(x_0-l)^2}{\exp\left\{\int_0^T {\varrho_t}'\varrho_t dt\right\}-1} - \dfrac{c d}{2}\int_0^T\int_t^T {\varrho_s}'\varrho_s dsdt\\
  & + \dfrac{c}{2}\int_0^T \ln(\det({\sigma_t}'\sigma_t))dt - \dfrac{c dT}{2}\ln(\mathcal{\pi} c).
\end{align*}
\end{proposition}

\begin{proof}

{\mou Fix $\varrho:=\{\varrho_t\}_t\in\mathfrak{\Xi}_0$.} {\mou For any fixed time $t\in[0,T]$}, the classical exploratory problem without model uncertainty has the optimal control density function

\begin{equation}
  \pi_t^*(v;\varrho) = \dfrac{\exp\left\{-\dfrac{1}{c}\left(\dfrac{1}{2}v'{\mou \sigma_t'\sigma_t} vV_{xx}+{\varrho_t'\mou\sigma_t} vV_x\right)\right\}}{\displaystyle\int_{\mathbb{R}^d}\exp\left\{-\dfrac{1}{c}\left(\dfrac{1}{2}v'{\mou\sigma_t'\sigma_t} vV_{xx}+{\mou\varrho_t'\sigma_t} vV_x\right)\right\} dv} \ , \label{explorOptPi}
\end{equation}
which follows a high-dimensional Gaussian distribution. Assume $\mathcal{V}$ is the Gaussian random variable whose probability density function at $t$ is {\mou $\pi_t^*(\cdot;\rho)$}. It is easy to verify that $\mathcal{V}\sim \mathcal{N}(v|\bm{q}(t), \bm{\Sigma}(t))$ where $\bm{q}(t) = -\sigma_t^{-1}\varrho_t\frac{V_x}{V_{xx}}$ and $\bm\Sigma(t) = ({\sigma_t}'\sigma_t)^{-1}\frac{c}{V_{xx}}$. More specifically, the optimal density function of the control process is
\begin{equation}
\begin{aligned}
  &\pi_t^*(v;\varrho) = (2\pi)^{-\frac{d}{2}}|\det\bm\Sigma(t)|^{-\frac{1}{2}}\exp\left\{-\dfrac{V_{xx}}{2c}\left(v-\bm{q}(t)\right)'{\sigma_t}'\sigma_t\left(v-\bm{q}(t)\right)\right\}\\
  &=\left(\dfrac{2\pi c}{V_{xx}}\right)^{-\frac{d}{2}}\sqrt{\det({\sigma_t}'\sigma_t)}\exp\left\{-\dfrac{V_{xx}}{2c}\left(v+\sigma_t^{-1}\varrho_t\dfrac{V_x}{V_{xx}}\right)'{\sigma_t}'\sigma_t\left(v+\sigma_t^{-1}\varrho_t\dfrac{V_x}{V_{xx}}\right)\right\}.
\end{aligned} \label{optimaldensity}
\end{equation}
Furthermore, the Shannon-entropy term at time $t$ is
\begin{equation*}
  -\int_{\mathbb{R}^d}{\mou\pi_t^*(v;\varrho)}\ln{\mou \pi_t^*(v;\varrho)}dv = \dfrac{1}{2}\ln\left((\dfrac{2\pi e c}{V_{xx}})^d\det({\sigma_t}'\sigma_t)^{-1}\right).
\end{equation*}

Plug \eqref{optimaldensity} back into \eqref{HJB4-0} with the fixed $\varrho$, \footnote{We take the following expectation of $\mathcal{V}$ or quadratic form of $\mathcal{V}$ under the optimal strategy's probability measure.}
\begin{align*}
  0 &= V_t + \dfrac{V_{xx}}{2}\mathbb{E}\left[(\mathcal{V}-\bm{q}(t))'{\sigma_t}'\sigma_t(\mathcal{V}-\bm{q}(t))\right]  -\dfrac{{\varrho_t}'\varrho_t V_x^2}{2V_{xx}}-\dfrac{c}{2}\ln\left(\left(\dfrac{2\pi e c}{V_{xx}}\right)^d\det({\sigma_t}'\sigma_t)^{-1}\right)\\
    &= V_t + \dfrac{V_{xx}}{2}\text{tr}({\sigma_t}'\sigma_t\bm\Sigma)-\dfrac{{\varrho_t}'\varrho_t V_x^2}{2V_{xx}}-\dfrac{c d}{2}\ln\dfrac{2\pi c}{V_{xx}}+\dfrac{c}{2}\ln(\det({\sigma_t}'\sigma_t))-\dfrac{c d}{2}\\
    &= V_t - \dfrac{{\varrho_t}'\varrho_t V_x^2}{2V_{xx}} - \dfrac{c d}{2}\ln\dfrac{2\pi c}{V_{xx}}+\dfrac{c}{2}\ln(\det({\sigma_t}'\sigma_t)).
\end{align*}

We guess the solution $V$ has the form $V(t,x)=A(t)x^2+B(t)x+C(t)$, so $\frac{V_x^2}{V_{xx}}=2A(t)x^2+2B(t)x+\frac{B^2(t)}{2A(t)}$ and it is clear to solve $A(t),B(t),C(t)$ by ODE systems and their solutions are
\begin{align*}
 A(t)= & \exp\left\{-\displaystyle\int_t^T {\varrho_s}'\varrho_s ds\right\}, \qquad
 B(t)= -2\omega\exp\left\{-\displaystyle\int_t^T {\varrho_s}'\varrho_s ds\right\},\\
 C(t)= & \omega^2\exp\left\{-\int_t^T {\varrho_s}'\varrho_s ds\right\}-(\omega-l)^2 - \dfrac{c d}{2}\int_t^T\int_s^T {\varrho_r}'\varrho_r drds   \\ &+\dfrac{c}{2}\int_t^T \ln(\det({\sigma_s}'\sigma_s))ds - \dfrac{c d}{2}\ln(\pi c)(T-t).
\end{align*}

Hence, the classical exploratory mean-variance problem has the following explicit solution for the value function
\begin{align*}
 &V(x,t) = (x-\omega)^2\exp\left\{-\int_t^T {\varrho_s}'\varrho_s ds\right\} - \dfrac{c d}{2}\int_t^T\int_s^T {\varrho_r}'\varrho_r drds \\
  &\qquad\qquad+ \dfrac{c}{2}\int_t^T \ln(\det({\sigma_s}'\sigma_s))ds - \dfrac{c d}{2}\ln(\pi c)(T-t) -(\omega-l)^2, \\
  &V_x = 2(x-\omega)\exp\left\{-\int_t^T {\varrho_s}'\varrho_s ds\right\},\qquad
  V_{xx} = 2\exp\left\{-\int_t^T {\varrho_s}'\varrho_s ds\right\}>0,\\
  &\bm{q}(t) = -\sigma_t^{-1}\varrho_t\frac{V_x}{V_{xx}}= -{\sigma_t}^{-1}\varrho_t(x-\omega),\\
  &\bm\Sigma(t) = ({\sigma_t}'\sigma_t)^{-1}\frac{c}{V_{xx}} = \frac{c}{2}({\sigma_t}'\sigma_t)^{-1}\exp\left\{\int_t^T {\varrho_s}'\varrho_s ds\right\}.
\end{align*}
A simple computation derives
\begin{align*}
 &\int_{\mathbb{R}^d} {\varrho_t}'\sigma_t v \pi^*_t(v;\varrho)dv = {\varrho_t}'\sigma_t\mathbb{E}[\mathcal{V}] = -{\varrho_t}'\varrho_t(x-\omega),\\
 &\int_{\mathbb{R}^d} v'{\sigma_t}'\sigma_t v \pi^*_t(v;\varrho)dv = \mathbb{E}[\mathcal{V}'{\sigma_t}'\sigma_t\mathcal{V}] = \mathbb{E}[(\mathcal{V}-\bm{q}(t))'{\sigma_t}'\sigma_t (\mathcal{V}-\bm{q}(t))] + \bm{q}(t)'{\sigma_t}'\sigma_t\bm{q}(t)\\
  &= \text{tr}({\sigma_t}'\sigma_t\bm\Sigma(t)) + (x-\omega)^2{\varrho_t}'{\sigma_t^{-1}}'{\sigma_t}'\sigma_t\sigma_t^{-1}\varrho_t = \dfrac{c d}{2}\exp\left\{\int_t^T {\varrho_s}'\varrho_s ds\right\} + (x-\omega)^2 {\varrho_t}'\varrho_t.
\end{align*}
The wealth dynamic \eqref{wealth} under the optimal control distribution \eqref{optimaldensity} becomes
\begin{align*}
  &dX_t^{\pi^*} = -{\varrho_t}'\varrho_t(X_t^{\pi^*}-\omega)dt + \sqrt{\dfrac{c d}{2}\exp\left\{\int_t^T {\varrho_s}'\varrho_s ds\right\} + (X_t^{\pi^*}-\omega)^2 {\varrho_t}'\varrho_t}\;dW_t,\\
 & X_t^{\pi^*} = x_0 - \int_0^t {\varrho_s}'\varrho_s(X_s^{\pi^*}-\omega)ds + \int_0^t \sqrt{\dfrac{c d}{2}\exp\left\{\int_s^T {\varrho_r}'\varrho_r dr\right\} + (X_s^{\pi^*}-\omega)^2 {\varrho_s}'\varrho_s}\;dW_s,
\end{align*}
\begin{align*}
 & \mathbb{E}[X_t^{\pi^*}] = x_0 - \int_0^t {\varrho_s}'\varrho_s(\mathbb{E}[X_s^{\pi^*}]-\omega)ds.
\end{align*}
{\mou The last equation above can be treated as an ODE for $\mathbb{E}[X_t^{\pi^*}]-\omega$}, whose solution provides the optimal Lagrangian multiplier $\omega=\frac{l\exp\left\{\int_0^T {\varrho_t}'\varrho_t dt\right\}-x_0}{\exp\left\{\int_0^T {\varrho_t}'\varrho_t dt\right\}-1}$. The value function at time 0 for the minimization problem \eqref{main0} is
\begin{align*}
  V(x_0,0) = M(\varrho,\pi^*(\cdot;\varrho))=&\dfrac{(x_0-l)^2}{\exp\left\{\int_0^T {\varrho_t}'\varrho_t dt\right\}-1} - \dfrac{c d}{2}\int_0^T\int_t^T {\varrho_s}'\varrho_s dsdt\\ &+ \dfrac{c}{2}\int_0^T \ln(\det({\sigma_t}'\sigma_t))dt - \dfrac{c dT}{2}\ln(\pi c).
\end{align*}
\end{proof}

Now we choose the worst $\varrho$ based on the solution of $M(\varrho,\pi^*(\cdot;\varrho))$. {\mou Regarding the representation formula of $M(\varrho, \pi^*(\cdot;\varrho))$, it is obvious that the worst case condition is given by for each $t\in[0,T]$
 \begin{equation}
  \varrho_t^* := \arg\min\limits_{\varrho_t\in \Xi} \Upsilon_t,\label{explorrho}
\end{equation}
where $\Upsilon_t:=\varrho_t'\varrho_t$. The minimum of $\Upsilon_t$ is attainable and it is denoted by $\Upsilon_t^*$. We note that $\Upsilon_t^*$ is strictly positive because $\Xi$ is a closed convex subset of $\mathbb R^{d\times d}\setminus \{0\}$. If $\varrho^*_t$ is given in \eqref{explorrho} for each $t\in[0,T]$, then $(\varrho^*, \pi^*(\cdot;\rho^*))$ is the solution of $\max\limits_{\varrho\in\mathfrak{\Xi}_0}\min\limits_{\pi\in\mathcal{A}_0} M(\varrho,\pi)$.}

\subsection{Robust solution and saddle point property} \label{RobustSolution}

Based on the solution of classical exploratory mean-variance problem, we inherit the same model setting and define a specific control distribution by
\begin{equation}
  \pi_t^0:=\pi_t^*(v;\varrho^*) \sim \mathcal{N}\left(v|-\sigma_t^{-1}\varrho_t^*(x-\omega), \dfrac{c}{2}e^{\int_t^T \Upsilon_s^* ds}({\sigma_t}'\sigma_t)^{-1}\right) =: \mathcal{N}(v|\bm{q}^0(t),\bm\Sigma^0(t)).
\end{equation}
Again, we assume $\mathcal{V}_0$ is a Gaussian random variable whose probability density function at time $t$ is $\pi_t^0$.

\begin{theorem}\label{minmaxtheorem}
  $(\varrho^*, \pi^*(\cdot;\varrho^*))$ defined in \eqref{explorOptPi} and \eqref{explorrho} is the solution of robust exploratory problem \eqref{main1} as well as a saddle point of $M(\varrho,\pi)$, i.e.,{\mou $\min\limits_{\pi\in\mathcal{A}_0}\max\limits_{\varrho\in\mathfrak{\Xi}_0} M(\pi,\varrho) = \max\limits_{\varrho\in\mathfrak{\Xi}_0}\min\limits_{\pi\in\mathcal{A}_0}M(\pi,\varrho)=M(\varrho^*,\pi^*(\cdot;\varrho^*))$.}
\end{theorem}

\begin{proof}
For any market parameter $\varrho=\{\varrho_t\}_t\in \mathfrak{\Xi}_0$, the wealth dynamic is
\begin{align*}
 \int_{\mathbb{R}^d} {\varrho_t}'\sigma_t v \pi^0_t(v)dv &= {\varrho_t}'\sigma_t\mathbb{E}[\mathcal{V}_0] = -{\varrho_t}'\varrho_t^*(x-\omega),\\
  \int_{\mathbb{R}^d} v'{\sigma_t}'\sigma_t v \pi^0_t(v)dv &= \mathbb{E}[{\mathcal{V}_0}'{\sigma_t}'\sigma_t\mathcal{V}_0] \\
   &= \mathbb{E}[(\mathcal{V}_0-\bm{q}^0(t))'{\sigma_t}'\sigma_t (\mathcal{V}_0-\bm{q}^0(t))] + \bm{q}^0(t)'{\sigma_t}'\sigma_t\bm{q}^0(t)\\
   &= \text{tr}({\sigma_t}'\sigma_t\bm\Sigma^0(t)) + (x-\omega)^2{\varrho_t^*}'\varrho_t^*\\
   &= \dfrac{c d}{2}e^{\int_t^T \Upsilon_s^* ds} + (x-\omega)^2 \Upsilon_t^*,
\end{align*}
and
\begin{equation*}
  dX_t^{\pi^0} = -{\varrho_t}'\varrho_t^*(X_t^{\pi^0}-\omega)dt + \sqrt{\dfrac{c d}{2}e^{\int_t^T \Upsilon_s^* ds} + (X_t^{\pi^0}-\omega)^2 \Upsilon_t^*}\;dW_t.
\end{equation*}
Apply It\'{o}'s formula to $F_t=(X_t^{\pi^0}-\omega)^2$,
\begin{align*}
  dF_t =& \left(\left[-2{\varrho_t}'\varrho_t^*+\|\varrho_t^*\|^2\right]F_t+\dfrac{c d}{2}e^{\int_t^T \Upsilon_s^* ds}\right) dt + 2\sqrt{F_t^2\|\varrho_t^*\|^2+\dfrac{c d}{2}e^{\int_t^T \Upsilon_s^* ds}}\;dW_t, \\
  F_0 =& (x_0-\omega)^2.
\end{align*}
Let $N_t = \mathbb{E}[F_t]$, $a(t)= -2{\varrho_t}'\varrho_t^*+\|\varrho_t^*\|^2$, $b(t)=\frac{c d}{2}\exp\{\int_t^T \Upsilon_s^* ds\}$, then $N_t$ satisfies the ODE $dN_t = a(t)N_t dt + b(t)dt$ whose solution is
\begin{equation}
\begin{aligned}
   N_t =& N_0\exp\left\{\int_0^t a(s)ds \right\}+\exp\left\{\int_0^t a(s)ds \right\}\int_0^t b(s)\exp\left\{-\int_0^s a(r)dr \right\}ds\\
  =& (x_0-\omega)^2\exp\left\{\int_0^t -2{\varrho_s}'\varrho_s^*+\|\varrho_s^*\|^2 ds \right\}+\dfrac{c d}{2} \exp\left\{\int_t^T \|\varrho_s^*\|^2 ds \right\} \\&\exp\left\{\int_0^t -2({\varrho_s}'\varrho_s^*+\|\varrho_s^*\|^2)ds\right\} \int_0^t \exp\left\{\int_0^s 2({\varrho_r}'\varrho_r^*-\|\varrho_r^*\|^2) dr \right\} ds \label{equationNt}.
\end{aligned}
\end{equation}
We can use a similar argument as Proposition \ref{Theorem41} and obtain
\begin{equation*}
  \omega = \dfrac{l\exp\left\{\int_0^T {\varrho_t}'\varrho_t^* dt\right\}-x_0}{\exp\left\{\int_0^T {\varrho_t}'\varrho_t^* dt\right\}-1}.
\end{equation*}
To compute {\mou $M(\varrho,\pi^0)$}, the entropy term is
\begin{align*}
  &c\int_0^T\int_{\mathbb{R}^d}\pi_t^0(v)\ln\pi_t^0(v)dv dt = -\dfrac{c T}{2}\ln((2\pi e)^d) - \dfrac{c}{2}\int_0^T \ln(|\det\bm\Sigma^0(t)|)dt\\
  &= -\dfrac{c Td}{2}\ln(\pi e c)+\dfrac{c T}{2}\ln(|\det({\sigma_t}'\sigma_t)|)-\dfrac{c d}{2}\int_0^T\int_t^T\|\varrho_s^*\|^2 ds dt.
\end{align*}
The terminal preference term is
\begin{align*}
 & \mathbb{E}[(X_T^{\pi^0}-\omega)^2]-(\omega-l)^2 = \dfrac{(x_0-l)^2\left(\exp\left\{\int_0^T\|\varrho_t^*\|^2 dt\right\}-1\right)}{\left(\exp\left\{\int_0^T{\varrho_t}'\varrho_t^* dt\right\}-1\right)^2} \\ &+\dfrac{c d}{2}\exp\left\{\int_0^T -2({\varrho_t}'\varrho_t^*-\|\varrho_t^*\|^2) dt \right\}\int_0^T \exp\left\{\int_0^t 2({\varrho_s}'\varrho_s^*-\|\varrho_s^*\|^2) ds \right\} dt.
\end{align*}
{\mou $M(\varrho,\pi^0)$} is the sum of the above two terms:
\begin{align*}
  M(\varrho,\pi^0)=& \dfrac{(x_0-l)^2\left(\exp\left\{\int_0^T\|\varrho_t^*\|^2 dt\right\}-1\right)}{\left(\exp\left\{\int_0^T{\varrho_t}'\varrho_t^* dt\right\}-1\right)^2}-\dfrac{c Td}{2}\ln(\pi e c)+\dfrac{c T}{2}\ln(|\det({\sigma_t}'\sigma_t)|) \\
  &+\dfrac{c d}{2}\exp\left\{\int_0^T -2({\varrho_t}'\varrho_t^*-\|\varrho_t^*\|^2) dt \right\}\int_0^T \exp\left\{\int_0^t 2({\varrho_s}'\varrho_s^*-\|\varrho_s^*\|^2) ds \right\} dt\\
  &-\dfrac{c d}{2}\int_0^T\int_t^T\|\varrho_s^*\|^2 ds dt.
\end{align*}
If the market uncertainty is achieved by $\varrho_t^*$, specially,
\begin{align*}
  M(\varrho^*,\pi^0)=& \dfrac{(x_0-l)^2}{\exp\left\{\int_0^T\|\varrho_t^*\|^2 dt\right\}-1}-\dfrac{c Td}{2}\ln(\pi e c)+\dfrac{c T}{2}\ln(|\det({\sigma_t}'\sigma_t)|)+\dfrac{c Td}{2}\\
  &-\dfrac{c d}{2}\int_0^T\int_t^T\|\varrho_s^*\|^2 ds dt.
\end{align*}
Then
\begin{align*}
  & M(\varrho^*,\pi^0) - M(\varrho,\pi^0)\\
  =& \dfrac{(x_0-l)^2 \left[\left(\exp\left\{\int_0^T{\varrho_t}'\varrho_t^* dt\right\}-1\right)^2-\left(\exp\left\{\int_0^T\|\varrho_t^*\|^2 dt\right\}-1\right)^2\right]}{\left(\exp\left\{\int_0^T\|\varrho_t^*\|^2 dt\right\}-1\right)\left(\exp\left\{\int_0^T{\varrho_t}'\varrho_t^* dt\right\}-1\right)^2}\\
  &+\dfrac{c d}{2}\left[T-\exp\left\{\int_0^T -2({\varrho_t}'\varrho_t^*-\|\varrho_t^*\|^2) dt \right\}\int_0^T \exp\left\{\int_0^t 2({\varrho_s}'\varrho_s^*-\|\varrho_s^*\|^2) ds \right\} dt\right].
\end{align*}
Since the convexity of $\Xi$ guarantees that ${\varrho_t}'\varrho_t^*\geq \|\varrho_t^*\|^2$ for any $\varrho_t\in\Xi$, the first term is non-negative. Furthermore, $\int_0^T\exp\left\{\int_0^t f(s) ds\right\}dt \leq T\exp\left\{\int_0^T f(t) dt\right\}$ holds true for any positive function $f$. As a result, both two terms are non-negative so $M(\varrho^*,\pi^0) \geq M(\varrho,\pi^0)$ holds true for all {\mou $\varrho\in\mathfrak{\Xi}_0$.}

So far we have
\begin{equation}
{\mou  M(\varrho^*,\pi^0)\geq \max\limits_{\varrho\in\mathfrak{\Xi}_0} M(\varrho, \pi^0)\geq \min\limits_{\pi\in\mathcal{A}_0}\max\limits_{\varrho\in \mathfrak{\Xi}_0} M(\pi,\varrho)\geq \max\limits_{\varrho\in \mathfrak{\Xi}_0}\min\limits_{\pi\in\mathcal{A}_0}M(\pi,\varrho)=M(\varrho^*,\pi^*(\cdot;\varrho^*)) \label{saddle},}
\end{equation}
and actually $\pi^0=\pi^*(\cdot;\varrho^*)$, so all the inequalities in \eqref{saddle} become equalities. This directly induces that the robust exploratory mean-variance problem has the solution {\mou$(\varrho^*, \pi^*(\cdot;\varrho^*))$}, and this is also the saddle point of $M(\pi,\varrho)$.
\end{proof}

The solution of the robust exploratory portfolio optimization is proved to be related to a min-max problem. A saddle point pair {\mou$(\varrho^*, \pi^*(\cdot;\varrho^*))$} reaches the equilibrium of the robust investor's market opinion and his/her investment behavior. The robust investor always has a conservative attitude towards the mean-return rate and is likely to reduce the investment on risky assets. Next we will analyze the effect of adding the robustness to portfolio strategies.

\section{Effect of robust strategies}  \label{section43}

Assume the genuine Sharpe ratio in risky market to be $\hat{\varrho}:=\{\hat{\varrho}_t\}_{t\in[0,T]}$. The discounted risky asset prices follow the dynamic
\[
dS_t^i = S_t^i\sum_{i=1}^d\sigma_{ij}(t)\left(\hat{\varrho_t}^jdt+dW_t^j\right)\text{ on $[0,T]$ for $i=1,\cdots,d$.}
\]
Due to the model uncertainty, an investor cannot precisely estimate $\hat{\varrho}$; instead, a misspecified Sharpe ratio $\varrho=\{\varrho_t\}_{t\in [0,T]}$ is chosen to determine his/her strategies. The robust investor further adjusts $\varrho_t$ to $\varrho^*_t$ as his/her worst-case perspective of market risk premium, as shown in Theorem \ref{minmaxtheorem}. We have in all four mean-variance portfolio management cases to compare in the following paragraphs: a misspecified investor with neither exploration nor robustness \cite{ZhouLi2000}; a misspecified investor with no exploration but robustness \cite{JinHanqing2015}; a misspecified investor with exploration but no robustness \cite{WangHaoran3}; a misspecified investor with both exploration and robustness. For simplicity, we consider the Sharpe ratio {\mou $\varrho$} and volatility matrix {\mou $\sigma$} are constants in this section.

\subsection{Misspecification and robustness without exploration} \label{sec431}

Previous literature \cite{ZhouLi2000, JinHanqing2015} have provided complete results when exploration is not involved. The optimal strategy of a misspecified investor is $v_t=\sigma^{-1}\varrho(\omega-X_t^\varrho)$. It can be verified that $\omega=\frac{le^{\varrho'\hat{\varrho}T}-x_0}{e^{\varrho'\hat{\varrho}T}-1}$ and the optimal terminal wealth distribution under {\mou model misspecification} is
\begin{equation*}
  X_T^\varrho = l + \dfrac{(x_0-l)(\exp\{-\frac{1}{2}\varrho'\varrho T-\varrho'W_T\}-1)}{\exp\{\varrho'\hat{\varrho}T\}-1}.
\end{equation*}
We further know that $\mathbb{E}[X_T^\varrho]=l$ : under model misspecification, the mean of the terminal wealth {\mou remains} invariant; however, the variance deviates from the optimal one and $Var[X_T^\varrho]=\frac{(x_0-l)^2(e^{\varrho'\varrho T}-1)}{(e^{\varrho'\hat{\varrho}T}-1)^2}\geq \frac{(x_0-l)^2}{e^{\hat{\varrho}'\hat{\varrho}T}-1}=Var[X_T^{\hat{\varrho}}]$. It concludes that an investor with a misspecified estimation of the Sharpe ratio $\varrho$ will suffer from a larger terminal variance than the correct model $\hat{\varrho}$.

A robust investor always uses a smaller Sharpe ratio $\varrho^*$ (in the sense of $L^2$ norm) to replace his/her original estimation $\varrho$. A direct effect of robustness is to compare $Var[X_T^{\varrho^*}]$ with $Var[X_T^\varrho]$. Fix any dimension $j\in\{1,2,...,d\}$, we split the Sharpe ratio adjustment into three cases:
\begin{itemize}
  \item $\hat{\varrho_j}<\varrho_j^*<\varrho_j$ : the investor overestimates the Sharpe ratio; the robust strategy enables the misspecified one to approach the genuine risk premium (though insufficiently) and reduce the terminal variance. The robust strategy is superior to the misspecified one.
  \item $\varrho_j^*<\varrho_j<\hat{\varrho_j}$ : the investor underestimates the Sharpe ratio; the robustness exaggerates the deviation and thus increases the terminal variance. In this case, the robust strategy is inferior to the misspecified one.
  \item $\varrho_j^*<\hat{\varrho_j}<\varrho_j$ : the investor overestimates the Sharpe ratio but overreacts during the risk premium adjustment. Whether the variance can be reduced depends on the comparison of $dist(\varrho_j^*, \hat{\varrho_j})$ and $dist(\hat{\varrho_j}, \varrho_j)$. Specially, if $dist(\varrho_j^*, \hat{\varrho_j})=dist(\hat{\varrho_j}, \varrho_j)$, we have $Var[X_T^{\varrho^*}]<Var[X_T^{\varrho}]$.
\end{itemize}

\begin{figure}[tbhp]
\centering
\subfloat[Variance curve 1d $\hat{\varrho}=0.5$]{\includegraphics[width=0.5\linewidth]{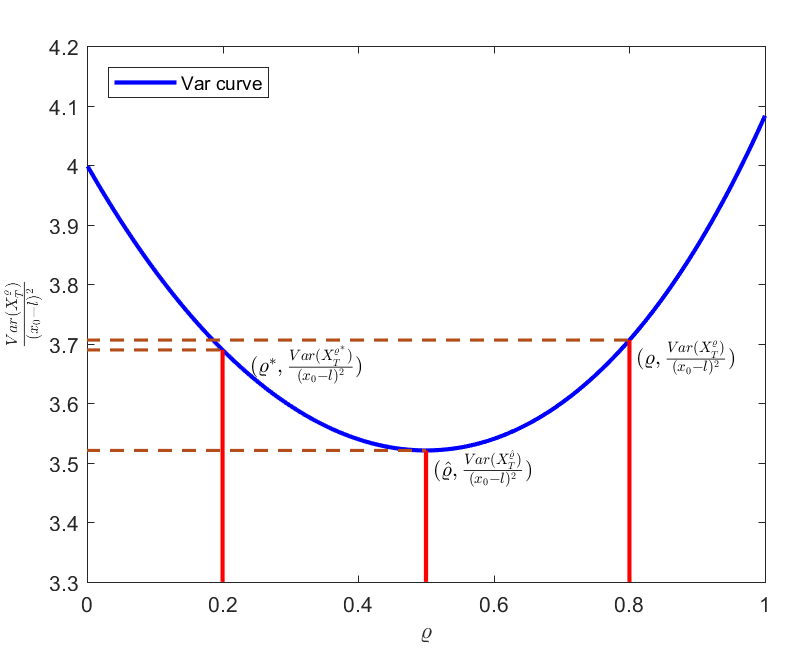}}
\subfloat[Variance contour 2d $\hat{\varrho}=(0.5,0.5)$]{\includegraphics[width=0.5\linewidth]{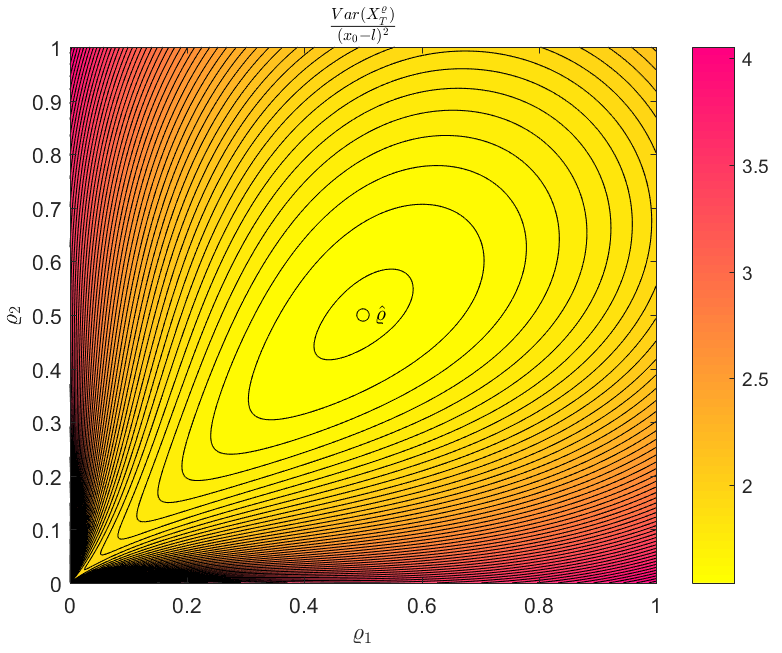}}
\caption{Variance reduction of misspecified risk premium}
\label{figure1-Chap4}
\end{figure}

The asymmetric variance structure mentioned in the last situation is illustrated as follows. Figure \ref{figure1-Chap4}(a) depicts the variance curve in the 1-dimension risk premium and parameters are set by $\varrho=0.8$, $\hat{\varrho}=0.5$, $\varrho^*=0.2$, $T=1$ and the curve is plotted for $\frac{Var[X_T^\varrho]}{(x_0-l)^2}=\frac{e^{\varrho'\varrho}-1}{(e^{\varrho'\hat{\varrho}}-1)^2}$. The risk premium reduction from $\varrho$ to $\varrho^*$ crosses the optimal value $\hat{\varrho}$ symmetrically, but the variance is still reduced slightly. The 2-dimensional variance contour of $\varrho$ in the unit square with the center $\hat{\varrho}=(0.5,0.5)$ is shown in Figure \ref{figure1-Chap4}(b). The variance surface reaches the basin when $\varrho=\hat{\varrho}$. When $\varrho$ deviates from $\hat{\varrho}$, the path from $\hat{\varrho}$ to the origin is flatter and the opposite direction is steeper. In conclusion, the variance contour (curve) leans to the side of smaller $\varrho$. Without any priori knowledge, reducing $\varrho$ is more likely to reduce variance than to increase variance. The robust strategy makes sense in accordance with the asymmetric variance structure.

\subsection{Misspecification and robustness with exploration}

It was shown in \cite{WangHaoran2} that the exploration does not affect the mean of the terminal wealth distribution, so $\mathbb{E}[X_T^\varrho]=l$ still holds. On the other hand, the variance is adjusted as the appearance of the exploration. Assume an investor precisely estimates the market risk premium to be $\hat{\varrho}$, the terminal variance is
\begin{equation}
\begin{aligned}
  Var[X_T^{\hat{\varrho}}] &= Var[X_T^{\hat{\varrho}}-\omega] \\
  &=\mathbb{E}[(X_T^{\hat{\varrho}}-\omega)^2] - \left(\mathbb{E}[X_T^{\hat{\varrho}}]-\omega\right)^2\\
  &=N_T^{\hat{\varrho}} - (l-\omega)^2 = \dfrac{(x_0-l)^2}{\exp\{\hat{\varrho}'\hat{\varrho}T\}-1} + \dfrac{cdT}{2}, \label{equa49}
\end{aligned}
\end{equation}
where $N_t^{\hat{\varrho}}$ and its solution is provided in \eqref{equationNt}. The terminal variance is increased by the term $\frac{cdT}{2}$ due to the exploration. If the investor selects a misspecified market scenario $\varrho$ instead of $\hat{\varrho}$, according to Proposition \ref{Theorem41}, the investor's policy is given as the Gaussian distribution
\begin{equation}
  \pi_t(v) \sim \mathcal{N}(-\sigma^{-1}\varrho'(x-\omega), \frac{c}{2}e^{\varrho'\varrho(T-t)}(\sigma'\sigma)^{-1}). \label{policy}
\end{equation}
We have the following Proposition to characterize the variance structure of $X_T^\varrho$, whose conclusion is slightly different from the case without exploration.

\begin{proposition}\label{Theorem43}
Assume the market risk premium is $\hat{\varrho}$ and an investor decides his/her strategy by a misspecified $\varrho$. Regard $Var[X_T^\varrho]$ as a function of $\varrho$, then $\exists \ k^*\in(\frac{1}{2}, 1)$ such that the terminal variance $Var[X_T^\varrho]$ attains a unique global minimum at $\varrho=k^*\hat{\varrho}$.
\end{proposition}

\begin{proof}
\textit{Step 1}.  Under {\mou the} misspecified risk premium $\varrho$, the investor's wealth process is given by
\begin{equation*}
  dX_t^\varrho = -\varrho'\hat{\varrho}(X_t^\varrho-\omega)dt + \sqrt{(X_t^\varrho-\omega)^2\varrho'\varrho+\dfrac{cd}{2}e^{\varrho'\varrho(T-t)}}dW_t.
\end{equation*}
We again represent $Var[X_t^\varrho]$ by the term of $N_t^\varrho=\mathbb{E}[(X_t^\varrho-\omega)^2]$. With the same argument as \eqref{equationNt}, $N_t^\varrho$ has {\mou its} dynamic and solution
\begin{align*}
 & dN_t^\varrho = (-2\varrho'\hat{\varrho}+\varrho'\varrho)N_t^\varrho dt + \dfrac{cd}{2}e^{\varrho'\varrho(T-t)}dt,\\
 & N_t^\varrho = \dfrac{cde^{\varrho'\varrho T}}{4(\varrho'\hat{\varrho}-\varrho'\varrho)}\left(e^{-\varrho'\varrho t}-e^{(\varrho'\varrho-2\varrho'\hat{\varrho})t}\right)+(x_0-\omega)^2e^{(\varrho'\varrho-2\varrho'\hat{\varrho})t}.
\end{align*}
Specially, when $\varrho=\hat{\varrho}$, the solution coincides with $N_T^{\hat{\varrho}}$. The optimal Lagrangian multiplier is given by $\omega = \frac{le^{\varrho'\hat{\varrho}T}-x_0}{e^{\varrho'\hat{\varrho}T}-1}$. Therefore, the variance under $\varrho$ is
\begin{equation}
\begin{aligned}
  Var[X_T^\varrho] &= N_T^\varrho - (l-\omega)^2 \\
  &=\dfrac{cd}{4(\varrho'\hat{\varrho}-\varrho'\varrho)}\left(1-e^{2(\varrho'\varrho-\varrho'\hat{\varrho})T}\right)+ (x_0-\omega)^2e^{(\varrho'\varrho-2\varrho'\hat{\varrho})T}-(l-\omega)^2\\
  &=\dfrac{cd(e^{2(\varrho'\varrho-\varrho'\hat{\varrho})T}-1)}{4(\varrho'\varrho-\varrho'\hat{\varrho})} +\left(\dfrac{(x_0-l)^2e^{\varrho'\hat{\varrho}T}}{e^{\varrho'\hat{\varrho}T}-1}\right)^2e^{(\varrho'\varrho-2\varrho'\hat{\varrho})T} -\left(\dfrac{x_0-l}{e^{\varrho'\hat{\varrho}T}-1}\right)^2\\
  &=\dfrac{cd(e^{2(\varrho'\varrho-\varrho'\hat{\varrho})T}-1)}{4(\varrho'\varrho-\varrho'\hat{\varrho})} +\dfrac{(x_0-l)^2(e^{\varrho'\varrho T}-1)}{(e^{\varrho'\hat{\varrho}T}-1)^2}. \label{Varterm}
\end{aligned}
\end{equation}
Again, $Var[X_T^\varrho]$ coincides with $Var[X_T^{\hat{\varrho}}]$ in \eqref{equa49} by taking limit $\varrho\rightarrow\hat{\varrho}$. Furthermore, a direct computation indicates that $Var[X_T^\varrho]$ is smooth at $\hat{\varrho}$. 

Next, we compute the gradient of $Var[X_T^\varrho]$ with respect to $\varrho$ to find the minimum point. Consider two terms in $Var[X_T^\varrho]$ separately,
\begin{equation*}
  \nabla_\varrho Var[X_T^\varrho] = \dfrac{cd}{2} \;\nabla_\varrho \left[\dfrac{e^{2(\varrho'\varrho-\varrho'\hat{\varrho})T}-1}{2(\varrho'\varrho-\varrho'\hat{\varrho})}\right] + (x_0-l)^2 \;\nabla_\varrho \left[\dfrac{e^{\varrho'\varrho T}-1}{(e^{\varrho'\hat{\varrho}T}-1)^2}\right].
\end{equation*}
The necessary condition of the exploration term to attain its minimum is
\begin{equation*}
  0=\nabla_\varrho \left[\dfrac{e^{2(\varrho'\varrho-\varrho'\hat{\varrho})T}-1}{2(\varrho'\varrho-\varrho'\hat{\varrho})}\right] =\left[\dfrac{Te^{2T(\varrho'\varrho-\varrho'\hat{\varrho})}}{\varrho'\varrho-\varrho'\hat{\varrho}} -\dfrac{e^{2T(\varrho'\varrho-\varrho'\hat{\varrho})}-1}{4(\varrho'\varrho-\varrho'\hat{\varrho})^2}\right](2\varrho-\hat{\varrho}),
\end{equation*}
the stationary point of which is $\varrho=\frac{1}{2}\hat{\varrho}$. Set the gradient of the classical variance term to be zero, i.e.
\begin{equation*}
  0 = \nabla_\varrho \left[\dfrac{e^{\varrho'\varrho T}-1}{(e^{\varrho'\hat{\varrho}T}-1)^2}\right] = \dfrac{2Te^{\varrho'\varrho T}}{(e^{\varrho'\hat{\varrho}T}-1)^2}\varrho - \dfrac{2Te^{\varrho'\hat{\varrho}T}(e^{\varrho'\varrho T}-1)}{(e^{\varrho'\hat{\varrho}T}-1)^3}\hat{\varrho},
\end{equation*}
the stationary point of which is $\varrho = \hat{\varrho}$. The stationary point of $Var[X_T^\varrho]$ is
\begin{equation}
  \varrho = \dfrac{cd\left(\dfrac{Te^{2T(\varrho'\varrho-\varrho'\hat{\varrho})}}{\varrho'\varrho-\varrho'\hat{\varrho}} -\dfrac{e^{2T(\varrho'\varrho-\varrho'\hat{\varrho})}-1}{4(\varrho'\varrho-\varrho'\hat{\varrho})^2}\right) + (x_0-l)^2\dfrac{2Te^{\varrho'\hat{\varrho}T}(e^{\varrho'\varrho T}-1)}{(e^{\varrho'\hat{\varrho}T}-1)^3}}{cd\left(\dfrac{Te^{2T(\varrho'\varrho-\varrho'\hat{\varrho})}}{\varrho'\varrho-\varrho'\hat{\varrho}} -\dfrac{e^{2T(\varrho'\varrho-\varrho'\hat{\varrho})}-1}{4(\varrho'\varrho-\varrho'\hat{\varrho})^2}\right)+(x_0-l)^2\dfrac{2Te^{\varrho'\varrho T}}{(e^{\varrho'\hat{\varrho}T}-1)^2}}\hat{\varrho}. \label{staypointcond}
\end{equation}
\eqref{staypointcond} indicates that the stationary point of $Var[X_T^\varrho]$ has the same direction as $\hat{\varrho}$.

\textit{Step 2}. For any given $\varrho=k\hat{\varrho}$, $k>0$, any rotation transformation from $\varrho$ to $\tilde{\varrho}$ where $\tilde{\varrho}\neq \varrho$ and $\|\tilde{\varrho}\| = \|\varrho\|$, we have $Var[X_T^{\tilde{\varrho}}]>Var[X_T^\varrho]$. This is because $\tilde{\varrho}'\hat{\varrho}<\varrho'\hat{\varrho}$, and thus $\varrho'\varrho-\varrho'\hat{\varrho}>\tilde{\varrho}'\tilde{\varrho}-\tilde{\varrho}'\hat{\varrho}$. The first term of \eqref{Varterm} is an increasing function of $\varrho'\varrho-\varrho'\hat{\varrho}$; the second term is a decreasing function of $\varrho'\hat{\varrho}$. As a result, $\tilde{\varrho}$ is always suboptimal to $\varrho$. Thus we can simplify the minimization problem by restricting on the line $\varrho=k\hat{\varrho}$ and find the optimal $k$ to minimize the variance instead.

\textit{Step 3}. Consider two terms in \eqref{Varterm} again.  $\frac{cd(e^{2(\varrho'\varrho-\varrho'\hat{\varrho})T}-1)}{4(\varrho'\varrho-\varrho'\hat{\varrho})}$ is a strictly convex function of $\varrho$. Therefore its stationary point $\varrho=\frac{1}{2}\hat{\varrho}$ attains its global minimum. For the second term restricts on the line $\varrho=k\hat{\varrho}$, $\frac{(x_0-l)^2(e^{\varrho'\varrho T}-1)}{(e^{\varrho'\hat{\varrho}T}-1)^2}=\frac{(x_0-l)^2(e^{k^2\|\hat{\varrho}\|^2 T}-1)}{(e^{k\|\hat{\varrho}\|^2T}-1)^2}$ is a strictly convex function of $k$ when $k>0$. The proof of two functions to be strictly convex is provided in Appendix \ref{AppendixB}. The stationary point $\varrho=\hat{\varrho}$ (or $k=1$) attains the global minimum of the second term. $Var[X_T^\varrho]$ is the sum of two strictly convex functions along $\varrho=k\hat{\varrho}$, so it is strictly convex as well and it has a unique global minimum. The stationary point \eqref{staypointcond} attains the global minimum and $k^*$ is the root of
\begin{equation*}
  \dfrac{cd\left(\dfrac{Te^{2T(k^2-k)\|\hat{\varrho}\|^2}}{(k^2-k)\|\hat{\varrho}\|^2} -\dfrac{e^{2T(k^2-k)\|\hat{\varrho}\|^2}-1}{4(k^2-k)^2\|\hat{\varrho}\|^4}\right) + (x_0-l)^2\dfrac{2Te^{k\|\hat{\varrho}\|^2T}(e^{k^2\|\hat{\varrho}\|^2 T}-1)}{(e^{k\|\hat{\varrho}\|^2T}-1)^3}}{cd\left(\dfrac{Te^{2T(k^2-k)\|\hat{\varrho}\|^2}}{(k^2-k)\|\hat{\varrho}\|^2} -\dfrac{e^{2T(k^2-k)\|\hat{\varrho}\|^2}-1}{4(k^2-k)^2\|\hat{\varrho}\|^4}\right)+(x_0-l)^2\dfrac{2Te^{k^2\|\hat{\varrho}\|^2 T}}{(e^{k\|\hat{\varrho}\|^2T}-1)^2}} = k.
\end{equation*}
Since the two terms of $Var[X_T^\varrho]$ has {\mou their} global minimum points $k=\frac{1}{2}$ and $k=1$ respectively, the global minimum point of $Var[X_T^\varrho]$ is a weight between its two nonnegative terms and thus $k^*\in(\frac{1}{2},1)$.
\end{proof}

Now we consider the effect on variance reduction of the robust strategy.

\begin{corollary}
For an investor taking the robust strategy to adjust his/her misspecified market viewpoint from $\varrho$ to $\varrho^*$, where $\varrho^*$ minimizes $L^2$ norm of $\varrho$ in the admissible set $\Xi$. Fix any individual asset $i$, then the conclusion is the same as the no exploration case in Subsection \ref{sec431}, except that the comparison object $\hat{\varrho}$ is replaced by $k^*\hat{\varrho}$:
\begin{itemize}
  \item $k^*\hat{\varrho_j}<\varrho_j^*<\varrho_j$ : the robust strategy helps reduce the terminal variance but not sufficiently. The robust strategy is superior to the misspecified one.
  \item $\varrho_j^*<\varrho_j<k^*\hat{\varrho_j}$ : the robustness exaggerates the deviation and thus increases the terminal variance. In this case, the robust strategy is inferior to the misspecified one.
  \item $\varrho_j^*<k^*\hat{\varrho_j}<\varrho_j$ : the robust strategy overreacts during the risk premium adjustment. Whether the variance can be reduced depends on the comparison of $dist(\varrho_j^*, k^*\hat{\varrho_j})$ and $dist(k^*\hat{\varrho_j}, \varrho_j)$. Specially, if $dist(\varrho_j^*, k^*\hat{\varrho_j})=dist(k^*\hat{\varrho_j}, \varrho_j)$, then we have $Var[X_T^{\varrho^*}]<Var[X_T^{\varrho}]$.
\end{itemize}
\end{corollary}

Figure \ref{figure2-Chap4} depicts the 2-dimensional variance contour of the exploratory mean-variance problem where the parameters are given by $\hat{\varrho}=(0.3,0.6)$, $c=0.5$, $T=1$, $l-x_0=0.3$. The red dot in Figure \ref{figure2-Chap4} is the optimal solution of the terminal variance term $\hat{\varrho}$, the blue dot is the optimal solution of the additional exploratory term in the terminal variance, whereas the minimal variance point is the green dot and $k^*\approx 0.683$ in this case. The dashed line going through the origin and $\hat{\varrho}$ always leads to the minimal variance direction and we can simplify the problem into merely searching this line.

\begin{figure}[tbhp]
\centering
\includegraphics[width=0.5\linewidth]{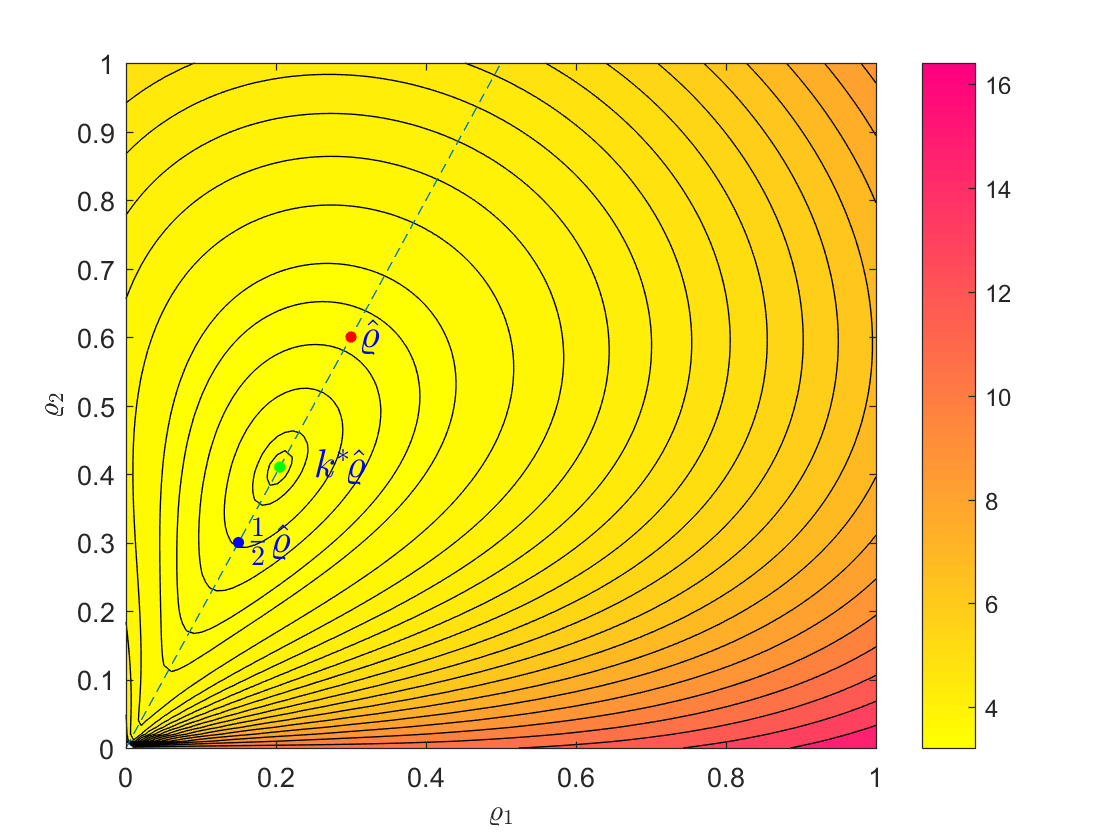}
\caption{2d variance contour in exploratory mean-variance problem}
\label{figure2-Chap4}
\end{figure}

\begin{remark}
\begin{itemize}
  \item[(a)] The involvement of exploration {\mou shifts} the minimum point of variance from $\hat{\varrho}$ to $k^*\hat{\varrho}$ where $\frac{1}{2}<k^*<1$. This indicates that based on the precise estimation of $\hat{\varrho}$, a minimal-variance targeted investor should be even further conservative; the market risk premium is rescaled by $k^*$. This phenomenon provides a reasonable explanation why robustness makes sense: when exploration is involved, the robust strategy with $\varrho$ has more chance to reduce the variance than with $\varrho^*$ due to the change of the minimum point from $\hat{\varrho}$ to $k^*\hat{\varrho}$. The behavior of the robust strategy naturally matches the target of minimizing the variance.
  \item[(b)] We know that the value function consists of the terminal variance and the entropy term. $\hat{\varrho}$ is optimal for the value function \eqref{main1} and $k^*\hat{\varrho}$ is the minimizer of the terminal variance. The risk premium adjustment from $\hat{\varrho}$ to $k^*\hat{\varrho}$ reduces the terminal variance but meanwhile deviates the optimality of the original problem \eqref{main1}. This is because $-\frac{cd}{4}\|\varrho\|^2T^2$ is a decreasing function of $\varrho$ in the entropy, and the entropy term increases as $\varrho$ decreases. We know that the terminal variance minimization is the effect of exploitation, and the entropy is the effect of exploration. An investor who searches for $k^*\!\hat{\varrho}$ instead of $\varrho^*$ essentially focus more on exploitation rather than exploration by rebalancing the weight between them.
\end{itemize}
\end{remark}

\section{Numerical experiments and results}\label{section44}

Having presented the theoretical formulation of the robust exploratory problem and the robust strategy against misspecification, now we focus on the real performance of the robust style investment under the exploratory background. In this section, first we simulate the wealth process \eqref{wealth} and compare the numerical behavior of the robust strategy against a misspecified one. Then we use the real SPX data to illustrate how robustness {\mou affects} exploration and parameter calibration.

\subsection{Performance comparison by wealth process simulation}

Given the uniform time mesh $\triangle t=T/n$ and {\mou the partition} $0=t_0<t_1<\dots<t_n=T$, the discrete wealth process of \eqref{wealth} is
\begin{equation}
  X_{i+1}^\varrho = X_i^\varrho + v_i' \sigma' \left[ \hat{\varrho} \triangle t + \triangle W_i\right],
\end{equation}
and $v_i$ is sampled from the distribution \eqref{policy} at $t=t_i$. We simulate $m$ trajectories and consider the investor to choose a misspecified $\varrho$ and the robust scenario $\varrho^*$ simultaneously. The actual market risk premium $\hat{\varrho}$ drives the wealth evolution but it is unknown to the investor.

\begin{algorithm}
\begin{algorithmic}
\State{Input initial endowment $x_0$, simulated paths $m$, time mesh $n$, convex set $\Xi$, misspecified $\varrho$, market parameters $\hat{\varrho}, \sigma$.}
\State{$\varrho^* \leftarrow$  Projection$(\varrho , \Xi)$}
\For{$i=0$ to $n-1$}
\For{$k=1$ to $m$}
\State{$\triangle W_i^k \leftarrow$  Simulate}
\State{$v_i^k(\varrho) \leftarrow$  PolicySampling$(\varrho)$}
\State{$v_i^k(\varrho^*)  \leftarrow$ PolicySampling$(\varrho^*)$}
\EndFor
\State{$X_{i+1}^\varrho \leftarrow$ Evolution$(X_i^\varrho, \varrho, \triangle W_i)$}
\State{$X_{i+1}^{\varrho^*} \leftarrow$ Evolution$(X_i^{\varrho^*}, \varrho^*, \triangle W_i)$}
\EndFor
\State{Var$(X_n^{\varrho}) \leftarrow$ Moments$(X_n^\varrho)$}
\State{Var$(X_n^{\varrho^*}) \leftarrow$ Moments$(X_n^{\varrho^*})$}
\\
\Return{Distribution and variance of $X_n^\varrho$ and $X_n^{\varrho^*}$}
\end{algorithmic}
\caption{Robust and misspecified policy simulations under exploration}
\label{algorithm41}
\end{algorithm}

Algorithm \ref{algorithm41} implements the simulation of the misspecified scenario and the robust strategy parallelly. Theoretically all the scenarios share the same terminal expectation $\mathbb{E}[X_T^\varrho]=\mathbb{E}[X_T^{\varrho^*}]=l$, so we directly compare the behavior of different scenarios by their variances. We choose the convex admissible set $\Xi$ to be two particular types: a cube and an elliptic. For the cube $\Xi:=\prod\limits_{j=1}^d [\underline{\varrho_j}, \overline{\varrho_j}]$ where the investor's estimation $\varrho\in\Xi$, the robust scenario is always regarded as $\varrho_j^*=\underline{\varrho_j}$. Hence, the projection of $\varrho$ in $\Xi$ is the vertex that all dimensions choose their left endpoint respectively. For the elliptic $\Xi:=\left\{\bar{\varrho} : \|\bar{\varrho}-\varrho\|\leq R, R<\|\varrho\|\right\}$ whose center is the misspecified scenario $\varrho$, the projection $\varrho^*=\varrho\,(1-\frac{R}{\|\varrho\|})$ keeps the direction invariant.

\begin{figure}[htb]
\centering
\subfloat[misspecified without robustness]{\includegraphics[width=0.5\linewidth]{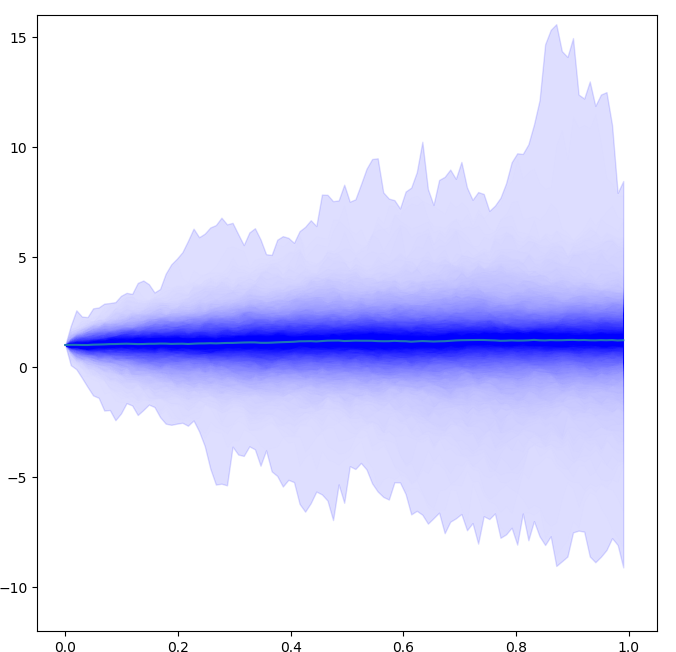}}
\subfloat[misspecified with robustness]{\includegraphics[width=0.5\linewidth]{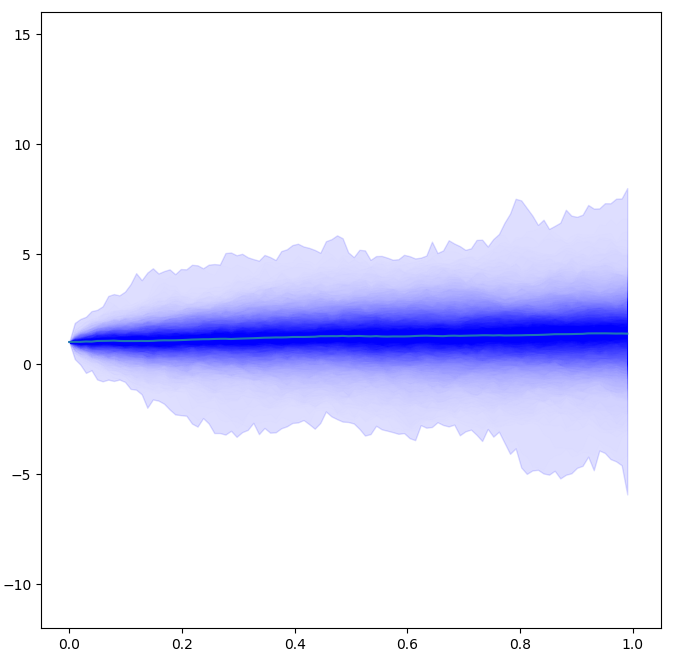}}
\caption{Simulation paths of exploratory wealth process $X_t^\varrho$ and $X_t^{\varrho^*}$}
\label{figure3-Chap4}
\end{figure}

Figure \ref{figure3-Chap4} depicts simulated paths of wealth {\mou processes} $X_t^\varrho$ and $X_t^{\varrho^*}$. In these two graphs, the darker the color, the more centralized the paths. It is obvious that the robust strategy (b) has a smaller envelope than the one without robustness (a). The parameters are set by: $x_0=1$, $l=1.2$, $T=1$, $d=4$, $c=1.5$, time mesh size $n=100$ and sample size $m=512$. We assume the diagonal volatilities $\Sigma_0=\text{diag}\,[0.15,0.2,0.4,0.3]$, and the correlation matrix
$$\rho = \left[\begin{matrix}  1& -0.85 & 0.45& 0.78 \\ -0.85& 1& -0.41& -0.62\\ 0.45& -0.41& 1& 0.64\\ 0.78& -0.62& 0.64& 1 \end{matrix}\right].$$
We can obtain the volatility matrix by taking the matrix square root of the variance matrix $\sigma'\sigma = \rho'\Sigma_0^2\rho$. It is further assumed that the correct risk premium $\hat{\varrho}=[0.4,0.4,0.4,0.4]$, but an investor chooses $\varrho=[0.5,0.5,0.5,0.5]$. The robust strategy replaces $\varrho$ by $\varrho^*=[0.25,0.25,0.25,0.25]$. Following Algorithm \ref{algorithm41}, the numerical values of two terminal variances are: $Var[X_T^\varrho]=3.503$, $Var[X_T^{\varrho^*}]=2.980$. In this implementation, we choose the identity risk premium among assets for convenience. Although after the robust adjustment $\varrho^*$ is further from $\hat{\varrho}$ than $\varrho$, it indeed reduces the variance. The expectation of simulated terminal wealth converges to target $l$ in both cases, which accords to the theoretical result.

\begin{table}[thbp]
{\footnotesize
\caption{Variance comparison of different scenarios}\label{table7}
\centering
\resizebox{\linewidth}{!}{
\begin{tabular}{c|c|c|c|c|c|c|c}
\hline \multirow{2}{*}{misspecified $\varrho$} & \multirow{2}{*}{$Var[X_T^\varrho]$} & \multicolumn{3}{c|}{cube convex set} & \multicolumn{3}{c}{elliptic convex set}\\
\cline{3-8}  & & $R$ & robust $\varrho^*$ & $Var[X_T^{\varrho^*}]$ & $R$ & robust $\varrho^*$ & $Var[X_T^{\varrho^*}]$ \\
\hline \multirow{4}{*}{[ 0.4, 0.5, 0.5 ,0.7 ]} & \multirow{4}{*}{ 3.664 } & 0.1 & [ 0.30, 0.40, 0.40, 0.60 ] & 3.313 & 0.2 & [ 0.330, 0.413, 0.413, 0.578 ] & 3.305 \\
 & & 0.2 & [ 0.20, 0.30, 0.30, 0.50 ] & 3.112 & 0.4 & [ 0.261, 0.326, 0.326, 0.456 ] & 3.090 \\
 & & 0.3 & [ 0.10, 0.20, 0.20, 0.40 ] & 2.909 & 0.6 & [ 0.191, 0.239, 0.239, 0.335 ] & 2.868 \\
 & & 0.35 & [ 0.05, 0.15, 0.15, 0.35 ] & 2.927 & 0.8 & [ 0.122, 0.152, 0.152, 0.213 ] & 2.949 \\
\hline \multirow{3}{*}{[ 0.15, 0.15, 0.35 ,0.4 ]} & \multirow{3}{*}{ 3.018 } & 0.05 & [ 0.10, 0.10, 0.30, 0.35 ] & 2.937 & 0.1 & [ 0.104, 0.104, 0.243, 0.278 ] & 2.843 \\
 & & 0.1 & [ 0.05, 0.05, 0.25, 0.3 ] & 2.958 & 0.2 & [ 0.058, 0.058, 0.136, 0.156 ] & 2.936 \\
 & & 0.15 & [ 0, 0, 0.20, 0.25 ] & 2.996 & 0.3 & [ 0.013, 0.013, 0.029, 0.034 ] & 3.019 \\
\hline \multirow{4}{*}{[ 0.5, 0.4, 0.3 ,0.2 ]} & \multirow{4}{*}{ 3.231} & 0.05 & [ 0.45, 0.35, 0.25, 0.15 ] & 3.218 & 0.2 & [ 0.315, 0.252, 0.189, 0.126 ] & 3.027 \\
 & & 0.1 & [ 0.40, 0.30, 0.20, 0.10 ] & 3.071 & 0.3 & [ 0.222, 0.178, 0.133, 0.089 ] & 2.986 \\
 & & 0.15 & [ 0.35, 0.25, 0.15, 0.05 ] & 3.112 & 0.4 & [ 0.130, 0.104, 0.078, 0.052 ] & 2.965 \\
 & & 0.2 & [ 0.30, 0.20, 0.10, 0 ] & 3.143 & 0.5 & [ 0.037, 0.030, 0.022, 0.015 ] & 3.096 \\
\hline
\end{tabular}}
}
\end{table}

More simulation results are presented in Table \ref{table7}. This time we keep the parameters the same as those in last experiment except $\hat{\varrho}=[0.2, 0.3, 0.4, 0.5]$ and larger sample size $n=8192$. The investor will apply different $\varrho$ and use different convex set $\Xi$ to describe the robust scenarios. In the cube type, $R$ is the distance from $\varrho$ to endpoints so $\underline{\varrho_j}=\varrho_j-R$. According to Table \ref{table7}, in most cases robustness reduces the variance. Furthermore, there is more mistake tolerance to adopt an underestimated risk premium than an overestimated one. Finally, the elliptic convex type is more sensitive than the cube type, and possibly approaches the basin of the variance surface even though the direction of $\varrho$ is far from $\hat{\varrho}$.

\subsection{Performance test by real market data}

Lastly, we consider the performance of the robust exploratory mean-variance portfolio optimization in a real market test. In order to discuss the behavior of robust/exploratory strategies in different market patterns, we choose two benchmarks separately: SPX daily data in US market (bull) and SSE Composite index daily data in Chinese market (bear) among last 15 years, see Figure \ref{dataplot}.

\begin{figure}[thbp]
\centering
\subfloat[discounted SPX daily price in US market]{\includegraphics[width=\linewidth]{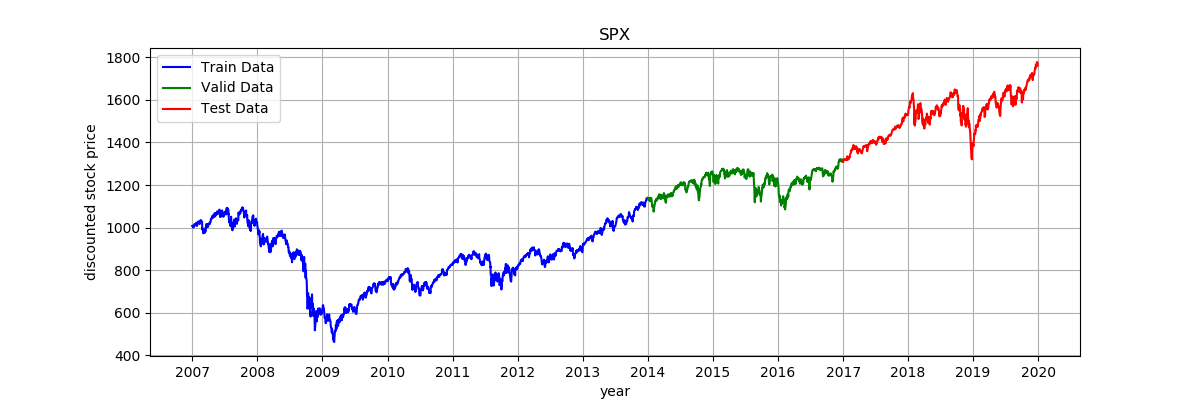}}\\
\subfloat[discounted SSE composite index in Chinese market]{\includegraphics[width=\linewidth]{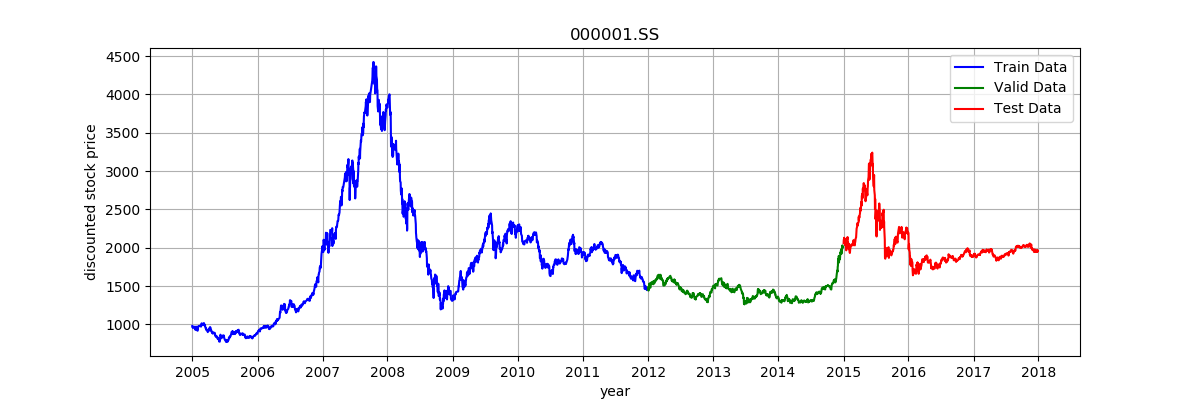}}
\caption{Real market data in different patterns}
\label{dataplot}
\end{figure}

The risk free rate is chosen {\mou to be} $r=0.02$ for price discounting. For convenience, we assume all the parameters to be calibrated are {\mou constants} and only 1 risky asset to be invested.

In each case, we split the stock price series into three parts: first 7 years to be train data, next 3 years to be valid data and last 3 years to be test data. We clip each data series with 1-year length successively and collect them to generate the data pool. The investment period is fixed to be $T=1$. SPX data has $n=252$ trading days per year in average while the number in SSE composite index is $n=243$. Unknown market parameters $\varrho$, $\sigma$ and $\omega$ are calibrated through training data and validation data. Finally we input the calibrated parameters into the test data and observe the investment performance.

Our calibration method as well as the portfolio optimization are based on the minimization problem \eqref{main0} and the optimal strategy distribution \eqref{policy}. The loss function of the minimization problem contains two terms: the terminal variance $\mathbb{E}[(X_T-\omega)^2]-(\omega-l)^2$ and the exploration loss equals to $-\frac{c}{2}\sum\limits_{i=0}^{n-1}\ln(\frac{\pi ec}{Var(v_i)}).$ {\mou The terminal variance is a function of $(\varrho, \omega)$ and the exploration loss is a function of $(\varrho, \sigma)$.}

Since the exploration loss is an increasing function of $\sigma$, it is ineffective to recover the volatility level of the real data by minimizing the exploratory value function \eqref{main0}. Instead, we estimate $\sigma$ by computing the historical volatility
\begin{equation*}
  \hat{\sigma}=\sqrt{n}\;\text{Std}(\ln(\dfrac{P_{i+1}}{P_i})),
\end{equation*}
where $P_i$ is the market price at day $i$ in $n$-length rolling window. Then we train the parameter $\varrho$ by the stochastic gradient descend scheme and update the optimal Lagrangian multiplier by $\omega\leftarrow\omega - lr(\mathbb{E}[X_T^{\pi^*}]-l)$, where $lr$ is the exponentially decaying learning rate defined by $lr=0.01e^{-0.0002k}$.

We further set the hyper-parameters by: batch size $m=512$, training steps $K=10000$, initial wealth $x_0=1$, target terminal wealth $l=1.2$, exploration intensity $c=0.001$. $c$ should be chosen properly so that it is able to keep balance between the terminal variance loss and the exploration loss.

\begin{algorithm}
\begin{algorithmic}
\State{Input $x_0$, $m$, $n$, $\Xi$, Train Data, Valid Data, Test Data.}
\State{Data pool $\leftarrow$ Data series}
\State{$\hat{\sigma}_{train}, \hat{\sigma}_{valid} \leftarrow$ HistVol(Train Data, Valid Data)}
\State{ Initialize $\rho,\omega$}
\For{$k=1$ to $K$}
\State{$\varrho_k^* \leftarrow$ Projection$(\varrho_k, \Xi)$}
\State{$P^k, \bar{P}^k \leftarrow$ EpochSampling(Train data pool, Test data pool)}
\For{$i=0$ to $n-1$}
\State{$v_i(\varrho_k) \leftarrow$ PolicySampling$(\varrho_k, \omega_k, \hat{\sigma}_{train})$}
\State{$X_{i+1}^{\varrho_k} \leftarrow X_i^{\varrho_k} + v_i(\varrho_k)(\frac{P_{i+1}^k}{P_i^k}-1)$}
\State{$v_i(\varrho_k^*) \leftarrow$ PolicySampling$(\varrho_k^*, \omega_k, \hat{\sigma}_{valid})$}
\State{$X_{i+1}^{\varrho_k^*} \leftarrow X_i^{\varrho_k^*} + v_i(\varrho_k^*)(\frac{\bar{P}_{i+1}^k}{\bar{P}_i^k}-1)$}
\EndFor
\State{Loss$(\varrho, \omega) \leftarrow$ Mean$((X_n^\varrho-\omega)^2)-(\omega-l)^2 +$ ExploreLoss$(\hat{\sigma}_{train}, \varrho)$}
\State{$\varrho_{k+1} \leftarrow$ AdamOptimizer(Loss, $\varrho$, lr)}
\State{$\omega_{k+1} \leftarrow \omega_k -$ lr(Mean$(X_n^{\varrho_k})-l$)}
\EndFor
\State{Mean$(X_n^{\varrho})$, Var$(X_n^{\varrho}) \leftarrow$ Moments$(X_n^\varrho(\omega, \hat{\sigma}_{valid}))$}
\State{Mean$(X_n^{\varrho^*})$, Var$(X_n^{\varrho^*}) \leftarrow$ Moments$(X_n^{\varrho^*}(\omega, \hat{\sigma}_{valid}))$}\\
\Return{Mean and variance of $X_n^\varrho$ and $X_n^{\varrho^*}$}
\end{algorithmic}
\caption{Robust and exploratory investment on real market data}
\label{algorithm42}
\end{algorithm}

Algorithm \ref{algorithm42} provides the performance test based on two real market data. Under the background of model uncertainty, it is suspected that the value $\varrho$ estimated from market data is still misspecified. A robust investor further cuts down the estimated value of $\varrho$ to $\varrho^*=R\varrho$ for his/her own market viewpoint where $R=0.8,0.6,0.4$ respectively. The key point is to compare the numerical results of different strategies (robust and no robust) in different market patterns (bull and bear). The investment performances are shown in Table \ref{table8}, Figure \ref{RESPX} and Figure \ref{RESSE}.

\begin{table}[thbp]
{\footnotesize
\caption{Market performance and calibration results of robust exploratory problem}
\label{table8}
\centering
\resizebox{\linewidth}{!}{
\begin{tabular}{c||c|c|c|c||c|c|c|c}
\hline benchmark & \multicolumn{4}{c||}{SPX} & \multicolumn{4}{c}{SSE}\\
\hline calibration & $\varrho$ & $\hat{\sigma}_{train}$ & $\hat{\sigma}_{valid}$ & $\omega$ & $\varrho$ & $\hat{\sigma}_{train}$ & $\hat{\sigma}_{valid}$ & $\omega$\\
\hline results & 1.104 & 2.025e-1 & 1.236e-1 & 1.418 & 1.220e-1 & 2.982e-1 & 1.721e-1 & 3.066 \\
\hline strategies & \multicolumn{3}{c}{Robust exploratory} & exploratory & \multicolumn{3}{c}{Robust exploratory} & exploratory \\
\hline R & 0.4 & 0.6 & 0.8 & 1.0 & 0.4 & 0.6 & 0.8 & 1.0 \\
$\varrho^*$   & 4.416e-1 & 6.623e-1 & 8.832e-1 & 1.104 & 4.881e-2 & 7.321e-2 & 9.762e-2 & 1.220e-1 \\
test loss     & 2.237e-1 & 2.053e-1 & 2.028e-1 & 2.156e-1 &-3.529 &-3.525 &-3.521 &-3.514 \\
test mean     & 1.126 & 1.193 & 1.252 & 1.301 &9.866e-1 &9.769e-1 &9.725e-1 &9.656e-1 \\
test variance & 4.285e-3 & 5.230e-3 & 4.499e-3 & 3.419e-3 &3.002e-3 &6.644e-3 &1.040e-2 &1.715e-2 \\
\hline
\end{tabular}}}
\end{table}

In the Chinese market, both valid region and test region basically go through the bear market, so the optimality of investment strategy under train data does not completely transfer to the test region. Despite following the optimal exploratory strategies, the market misspecification leads to negative profit in mean. In this situation, robustness helps reduce both the risk and the profit loss simultaneously. As a result, in both mean and variance, the robust strategy outperforms the purely exploratory strategy. On the contrary, the US market has gone through a remarkable long run bull. The robust strategy seems to be too conservative to adapt to the US market because in both mean and variance it underperforms the purely exploratory strategy.

In conclusion, whether robustness makes sense is most likely related to the financial market background. Compared with exploratory strategy, robustness takes advantage in bear market and disadvantage in bull market. A conservative investor adopts the robust strategy rather than a pure exploratory one in order to resist the downside risk, even though he will potentially miss the high profit in uptrend market.

\begin{figure}[thbp]
\centering
\subfloat[mean of wealth process trajectories]{\includegraphics[width=0.5\linewidth]{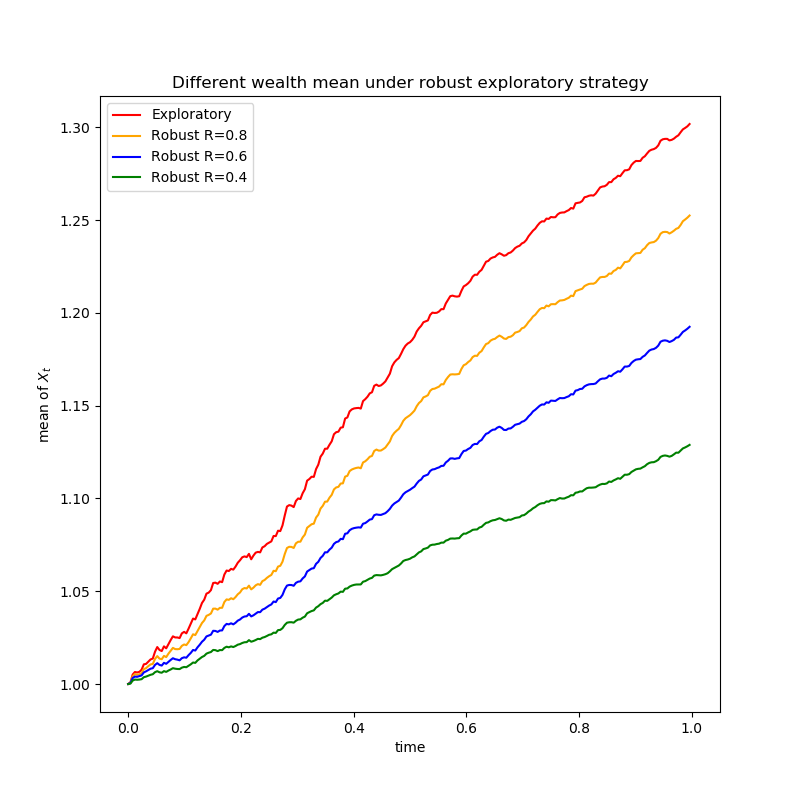}}
\subfloat[variance of wealth process trajectories]{\includegraphics[width=0.5\linewidth]{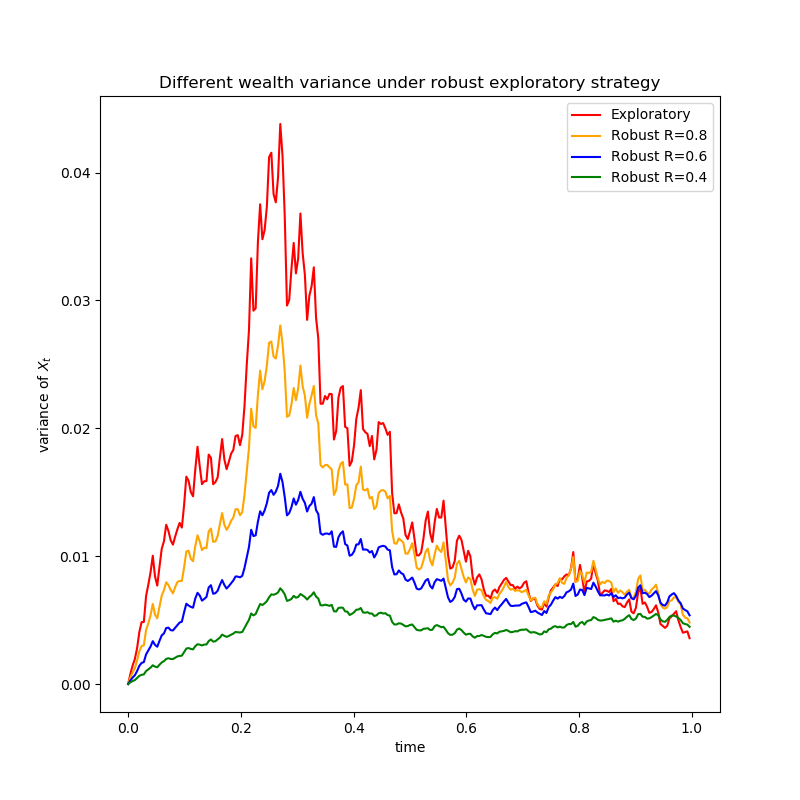}}
\caption{Robust-exploratory investment performances on SPX}
\label{RESPX}
\end{figure}
\begin{figure}[thbp]
\centering
\subfloat[mean of wealth process trajectories]{\includegraphics[width=0.5\linewidth]{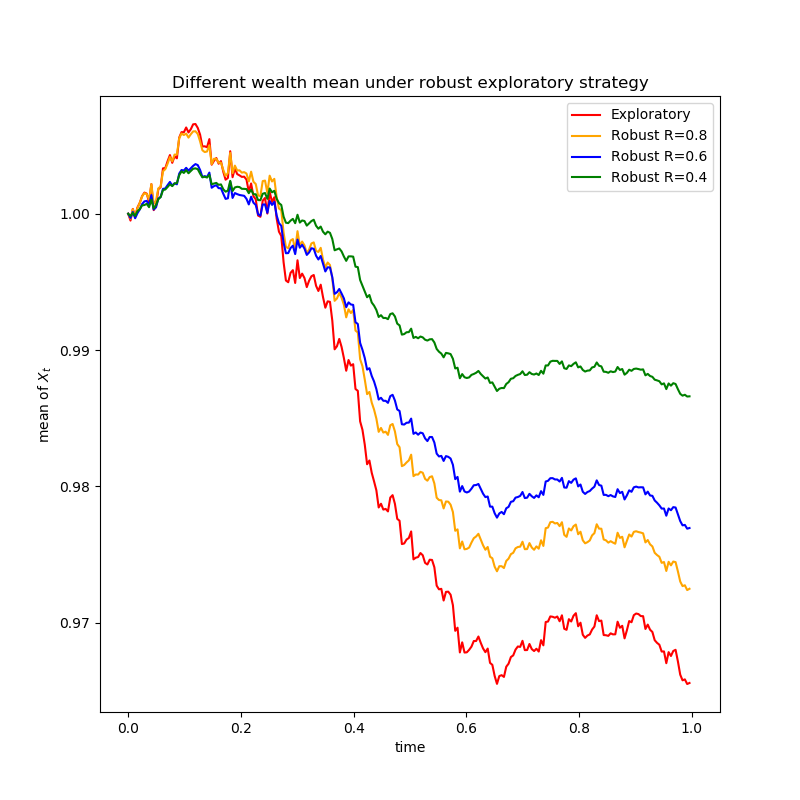}}
\subfloat[variance of wealth process trajectories]{\includegraphics[width=0.5\linewidth]{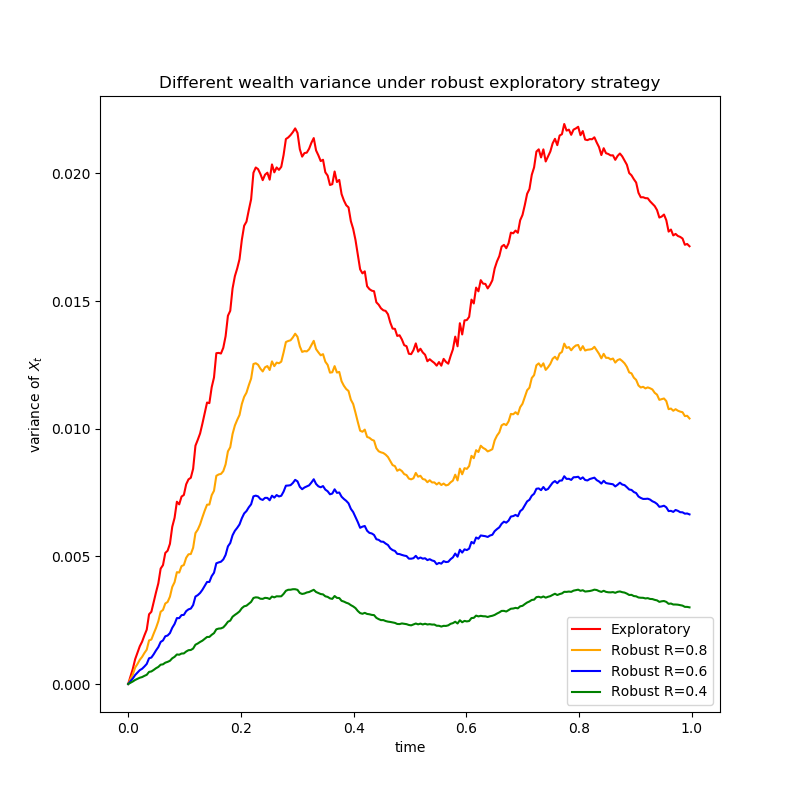}}
\caption{Robust-exploratory investment performances on SSE}
\label{RESSE}
\end{figure}

\section{Conclusion}\label{section45}

The robust exploratory mean-variance analysis was investigated in this paper. Based on the previous work of exploratory mean-variance problem, we inherited the exploratory wealth dynamic setting but further assumed there is model uncertainty in drift. Under the background of misspecification, a robust investor who always considers the worst scenario should seek for the sharpe ratio with minimal $L^2$ norm in the specific admissible set as his/her market perspective, and this is also a saddle-point of the min-max problem. Theoretically, an investment under a misspecified $\varrho$ deviates the target of minimizing the terminal variance, while the robust viewpoint $\varrho^*$ has more opportunity to reduce the deviation than to exaggerate it. Market parameter calibration in exploratory mean-variance optimization is based on the balance between exploration and exploitation. It can be justified that a robust investor's behavior is equivalent to transferring additional weight from exploration to exploitation. Finally, financial data backtests show that robustness outperforms the pure exploration and helps resist the downside risk in a bear market, while it underperforms in a bull market.

\appendix
\section{The proof of convexity in Proposition \ref{Theorem43}} \label{AppendixB}
\begin{lemma} \label{lemmaA1}
{\mou Let $f : \mathbb{R}\rightarrow \mathbb{R}$ and $g : \mathbb{R}^d \rightarrow \mathbb{R}$ be smooth and convex functions. Assume $f'>0$ and $f''>0$ on $\mathbb R^d$}. {\mou Assume also $g$ satisfies either $\nabla^2g>0$ or $\nabla g$ is invertible}. Then {\mou$f(g(x)):\mathbb R^d\to\mathbb R$} is strictly convex.
\end{lemma}

\begin{proof}
Compute the Hessian matrix of $f(g(x))$:
\begin{equation*}
  \nabla^2 f(g(x)) = \nabla ( f'(g(x))\nabla g(x) ) = f''(g(x))\nabla g(x)\nabla g(x)' + f'(g(x))\nabla^2 g(x).
\end{equation*}
{\mou With the assumptions on $f$ and $g$, we have, for arbitrary $y\in\mathbb{R}^d$, $$y'\nabla^2 f(g(x))y=f''(g(x)) \|y'\nabla g(x)\|^2 + f'(g(x)) y'\nabla^2g(x) y>0.$$ Then $\nabla^2 f(g(x))$ is also positive definite, which implies $f(g(x))$ is strictly convex.}
\end{proof}

\begin{theorem}
  The function $\frac{(e^{2(\varrho'\varrho-\varrho'\hat{\varrho})T}-1)}{2(\varrho'\varrho-\varrho'\hat{\varrho})T}$ is a strictly convex function of $\varrho$.
\end{theorem}

\begin{proof}
Let $g(x): = 2T(x'x-x'x_0)$ where $x,x_0\in\mathbb{R}^{d,+}$ and $x_0$ is fixed. Let $f(x):=\frac{e^x-1}{x}$ when $x\neq 0$ and $f(x) = 1$ when $x=0$. {\mou Then
$\nabla g(x)=4TI_d$ is invertible and $\nabla^2g(x)\equiv 0$.} We also have the derivative of $f(x)$
\begin{equation*}
  f'(x) = \left\{\begin{array}{ll} \dfrac{(x-1)e^x+1}{x^2} & \text{if} \quad x\neq 0 \\ \dfrac{1}{2} & \text{if} \quad x=0.
  \end{array} \right.
\end{equation*}
{\mou Let $h(x):=(x-1)e^x+1$, $h'(x)= xe^x$ has the critical point $x=0$, and $h(0)=0$ attains the global minimum. So $f'(x)>0$ for all $x\neq 0$, together with $f'(0)=\frac{1}{2}>0$ we have $f$ is strictly increasing.} The second order derivative of $f$ is given by
\begin{equation*}
 f''(x)=\left\{\begin{array}{ll} \dfrac{(x^2-2x+2)e^x-2}{x^3} & \text{if} \quad x\neq 0 \\ \dfrac{1}{3} & \text{if} \quad x=0.
  \end{array} \right.
\end{equation*}
Again let $l(x) := (x^2-2x+2)e^x-2$. Then $l'(x)=x^2e^x\geq 0$, which imply $l(x)$ is strictly. Together with $l(0)=0$, we have $l(x)<0$ when $x<0$, and $l(x)>0$ when $x>0$. Then $f''(x)=l(x)/x^3>0$ when $x\neq 0$, and $f''(0)=\frac{1}{3}>0$ as well, so $f$ is strictly convex. Applying Lemma \ref{lemmaA1}, $f(g(\varrho))=\frac{(e^{2(\varrho'\varrho-\varrho'\hat{\varrho})T}-1)}{2(\varrho'\varrho-\varrho'\hat{\varrho})T}$ is strictly convex.
\end{proof}

\begin{lemma}  \label{lemmaA2}
{\mou Let $f : \mathbb{R}^+\rightarrow \mathbb{R}$ and $g : \mathbb{R}^+\rightarrow \mathbb{R}^+$ be smooth functions. Assume $f''<0$ and $f'> 0$ on $\mathbb R^+$. Assume also $(f\circ g)''>0$ on $\mathbb R^+$. Then $g$ is strictly convex.}
\end{lemma}

\begin{proof}
The strict convexity of function $f(g(x))$ implies
\begin{equation*}
  f(g(x))'' = f''(g(x))(g'(x))^2 + f'(g(x))g''(x)>0
\end{equation*}
whereas $f''(g(x))<0$ and $f'(g(x))>0$. Then $g''(x)>0$ and thus $g$ is strictly convex.
\end{proof}

\begin{theorem} \label{theoremA3}
  The function $\frac{(e^{k^2\|\hat{\varrho}\|^2 T}-1)}{(e^{k\|\hat{\varrho}\|^2T}-1)^2}$ is a strictly convex function of $k$ {\mou on $\mathbb R^+$}.
\end{theorem}

\begin{proof}
It is equivalent to prove $g(x)=\frac{e^{x^2}-1}{(e^x-1)^2}$ is a strictly convex function {on \mou$\mathbb R^+$}. Choosing $f(x)=\ln x$, then $f(g(x))=\ln(e^{x^2}-1)-2\ln(e^x-1)$. Denote its second order derivative by
\begin{equation}
h(x):=f(g(x))''=\dfrac{2e^{x^2}}{e^{x^2}-1}-\dfrac{4x^2e^{x^2}}{(e^{x^2}-1)^2}+\dfrac{2e^x}{(e^x-1)^2}.   \label{hx}
\end{equation}
We note that $h(x)$ can be decomposed as $h(x)=2(h_1(x)+h_2(x))$, where $h_1(x)=\frac{e^{x^2}}{e^{x^2}-1}-\frac{x^2e^{x^2}}{(e^{x^2}-1)^2}$ and $h_2(x)=\frac{e^x}{(e^x-1)^2}-\frac{x^2e^{x^2}}{(e^{x^2}-1)^2}$. We can consider the two terms separately. For $h_1(x)$, we have on $\mathbb R^+$
\begin{align}
  h_1(x) =& \dfrac{e^{x^2}(e^{x^2}-1-x^2)}{(e^{x^2}-1)^2} = 1+\dfrac{(1-x^2)e^{x^2}-1}{(e^{x^2}-1)^2} \label{h1x},\\
\nonumber  h_1'(x) =& \dfrac{-2x^3e^{x^2}(e^{x^2}-1)-((1-x^2)e^{x^2}-1)2xe^{x^2}}{(e^x-1)^2} = \dfrac{2x^3e^{x^2}}{(e^x-1)^2} >0,
\end{align}
so $h_1(x)$ is an increasing function {\mou on $\mathbb R^+$}. Therefore, $\lim\limits_{x\rightarrow 0^+}h_1(x)=\frac{1}{2}$ is the lower bound of $h_1(x)$ and $\lim\limits_{x\rightarrow +\infty}h_1(x)=1$ is the upper bound of $h_1(x)$ on $\mathbb{R}^+$.
We further define $u_0(x):=\frac{x^2e^x}{(e^x-1)^2}$. Then the second term $h_2(x)$ can be represented by
\begin{align}
  h_2(x) &= \dfrac{1}{x^2}(u_0(x)-u_0(x^2)) \label{h2x},\\
\label{du0x}  u_0'(x) &= \dfrac{2xe^x}{(e^x-1)^2}-\dfrac{2x^2e^{2x}}{(e^x-1)^3}+\dfrac{x^2e^x}{(e^x-1)^2}=\dfrac{xe^x[e^x(2-x)-(2+x)]}{(e^x-1)^3}<0\,\, \text{on}\,\, \mathbb R^+,\\
\nonumber  u_1(x) &:= e^x(2-x)-(2+x), \\
\nonumber  u_1'(x) &= (1-x)e^x-1<0 \ \text{on}\ \mathbb{R}^+,\\
\nonumber  u_1''(x) &=-xe^x < 0 \ \text{on}\ \mathbb{R}^+,
\end{align}
thus $u_1(x)$ and $u_1'(x)$ are strictly decreasing on $\mathbb R^+$, $u_1(0)=0$, $u_1'(0)=0$, $u_1(x)<0$ on $\mathbb{R}^+$, $u_0(x)$ is strictly decreasing, $u_0(0)=0$ and $u_0(x)<0$ on $\mathbb R^+$. In order to satisfy $h_2(x)+h_1(x)>0$ conditioning on $h_1(x)>\frac{1}{2}$ when $x>0$, it suffices to prove that $h_2(x)\geq-\frac{1}{2}$, which is, $u_0(x)-u_0(x^2)+\frac{1}{2}x^2\geq 0$ on $\mathbb R^+$. We let $F_0(x):=u_0(x)-u_0(x^2)+\frac{1}{2}x^2$, its derivative is $F_0'(x)=u_0'(x)-2xu_0'(x^2)+x$. It is known that $u_0'(x^2)<0$ for any $x\in \mathbb R^+$, so the second term in $F_0'(x)$ is positive on $\mathbb R^+$. We then define $F_1(x):=u_0'(x)+x$ and claim that $F_1(x)\geq0$ on $\mathbb R^+$. According to \eqref{du0x}, the claim is equivalent to
\begin{equation}\label{claimA5}
  e^x[e^x(2-x)-(2+x)]+(e^x-1)^3\geq0\,\, \text{on}\,\, \mathbb R^+.
\end{equation}
It is convenient to do the change of variable $y:=e^x$ and thus the left hand side of \eqref{claimA5} becomes
$$\xi(y):=y[y(2-\ln y)-(2+\ln y)]+(y-1)^3\, \text{ for $y\in (1,+\infty)$}.
$$
Next, we compute the derivatives of $\xi(y)$ and we have on $(1,+\infty)$
\begin{equation*}
  \xi'(y) = 3y(y-1)-(2y+1)\ln y>0.
\end{equation*}
Together with $\xi(1)=0$, we have $\xi(y)>0$ on $(1,+\infty)$ and thus the claim $F_1(x)\geq 0$ on $\mathbb R^+$ holds. Therefore, $F_0'(x)=F_1(x)-2xu_0'(x^2)>0$ on $\mathbb R^+$. Together with $F_0(0)=0$, we have $F_0(x)>0$. It implies that $h_2(x)=\frac{F_0(x)}{x^2}-\frac{1}{2}>-\frac{1}{2}$ on $\mathbb R^+$ and thus $h(x)$ defined in \eqref{hx} is positive on $\mathbb R^+$. According to Lemma \ref{lemmaA2}, we conclude that $g$ is strictly convex and the proof is completed. The graph of function $h(x), h_1(x), h_2(x)$ are plotted in Figure \ref{AppendixBgraph} (b).
\end{proof}

\begin{remark}
$g$ has second order derivative
\begin{equation*}
  g''(x) = \dfrac{(4x^2+2)e^{x^2}(e^x-1)^2-2e^x(e^{x^2}-1)(e^x-1)+6e^{2x}(e^{x^2}-1)-8xe^xe^{x^2}}{(e^x-1)^4}
\end{equation*}
and $g''(0)=0$. However, it is rather tedious to prove $g''(x)>0$ on $\mathbb R^+$. Instead, we plot the graph of $g''(x)$ in Figure \ref{AppendixBgraph} (a).
\end{remark}

\begin{figure}[thbp]
\centering
\subfloat[$g''(x)$]{\includegraphics[width=0.5\linewidth]{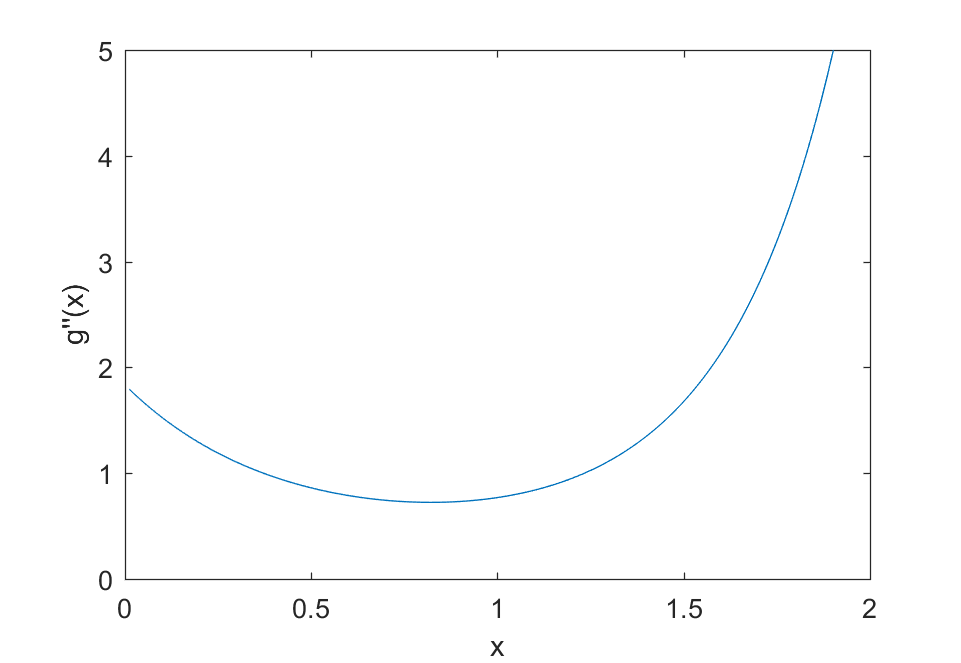}}
\subfloat[$h(x)$ and its accessory functions]{\includegraphics[width=0.45\linewidth]{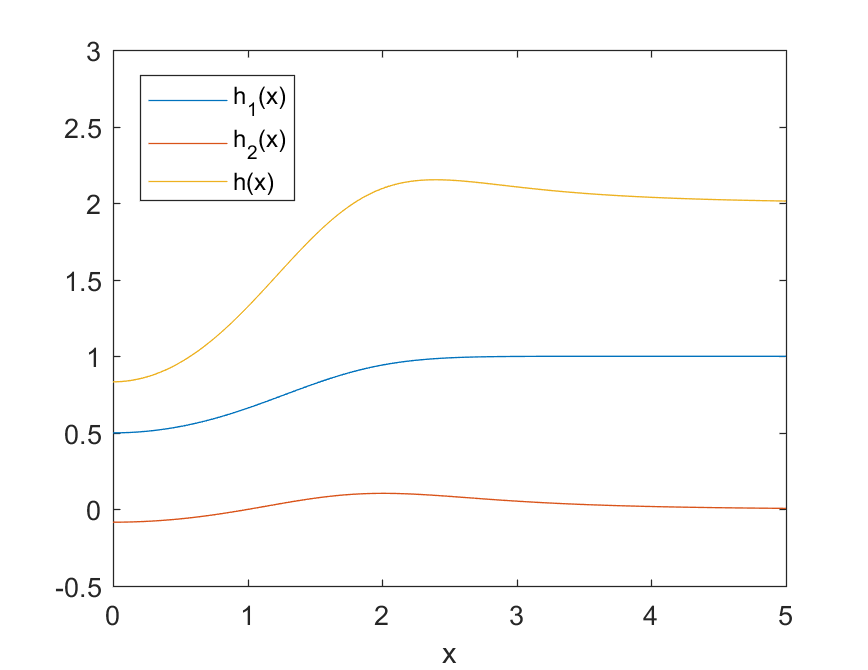}}
\caption{graph of functions in Theorem \ref{theoremA3}}
\label{AppendixBgraph}
\end{figure}


\begin{thebibliography}{10}

\bibitem{Avellaneda1995}
{\sc M.~Avellaneda, A.~Levy, and A.~Par\'{a}s}, {\em Pricing and hedging
  derivative securities in markets with uncertain volatilities}, Applied
  Mathematical Finance, 2 (1995), pp.~73--88.

\bibitem{Berestycki2002}
{\sc H.~Berestycki, J.~Busca, and Florent}, {\em Asymptotics and calibration of
  local volatility models}, Quantitative Finance, 2 (2002), pp.~61--69.

\bibitem{ChenEpstein2002}
{\sc Z.~Chen and L.~Epstein}, {\em Ambiguity, risk and asset returns in
  continuous time}, Econometrica, 70 (2002), pp.~1403--1443.

\bibitem{Goldfarb2003}
{\sc D.Goldfarb and G.Iyengar}, {\em Robust portfolio selection problems},
  Mathematics of Operations Research, 28 (2003), pp.~1--200.

\bibitem{Gundal2005}
{\sc A.~Gundel}, {\em Robust utility maximization for complete and incomplete
  market models}, Finance and Stochastics, 9 (2005), pp.~151--176.

\bibitem{Hansen2001}
{\sc L.~Hansen and T.~J.Sargent}, {\em Robust control and model uncertainty},
  American Economic Review, 91 (2001), pp.~60--66.

\bibitem{JinHanqing2015}
{\sc H.~Jin and X.~Zhou}, {\em Continuous-time portfolio selection under
  ambiguity}, Mathematical Control and Related Fields, 5 (2015), pp.~475--488.

\bibitem{LinRiedel2014}
{\sc Q.~Lin and F.~Riedel}, {\em Optimal consumption and portfolio choice with
  ambiguity}, Economic Theory, 71 (2021), pp.~1189--1202.

\bibitem{Lyons1995}
{\sc T.~Lyons}, {\em Uncertain volatility and the risk free synthesis of
  derivatives}, Applied Mathematical Finance, 2 (1995), pp.~117--133.

\bibitem{Matoussi2015}
{\sc A.~Matoussi, D.~Possama\'{i}, and C.~Zhou}, {\em Robust utility
  maximization in non-dominated models with 2bsdes: the uncertainty volatility
  model}, Mathematical Finance, 25 (2012), pp.~258--287.

\bibitem{Merton1980}
{\sc R.~C. Merton}, {\em On estimating the expected return on the market: An
  exploratory investigation}, Journal of Financial Economics, 8 (1980),
  pp.~323--361.

\bibitem{Skiadas2003}
{\sc C.~Skiadas}, {\em Robust control and recursive utility}, Finance and
  Stochastics, 7 (2003), pp.~475--489.

\bibitem{Tevzadze2013}
{\sc R.~Tevzadze, T.~Toronjadze, and T.~Uzunashvili}, {\em Robust utility
  maximization for a diffusion market model with misspecified coefficients},
  Finance and Stochastics, 17 (2013), pp.~535--563.

\bibitem{WangHaoran3}
{\sc H.~Wang}, {\em Large scale continuous-time mean-variance portfolio
  allocation via reinforcement learning}, arXiv:1907.11718v2,  (2019).

\bibitem{WangHaoran1}
{\sc H.~Wang, T.~Zariphopoulou, and X.~Zhou}, {\em Reinforcement learning in continuous time and space: A stochastic control approach},
  Journal of Machine Learning Research, 21 (2020), pp.~1--34.

\bibitem{WangHaoran2}
{\sc H.~Wang and X.~Zhou}, {\em Continuous-time mean-variance portfolio selection: A reinforcement learning framework}, Mathematical Finance, 30 (2020), pp.~1273--1308.

\bibitem{ZhouLi2000}
{\sc X.~Zhou and D.~Li}, {\em Continuous-time mean variance portoflio
  selection: A stochastic lq framework}, Applied Mathematics and Optimization,
  42 (2000), pp.~19--33.

\end{thebibliography}
\end{document}